\documentclass[oneside,11pt,reqno]{amsart}
\usepackage[utf8]{inputenc}
\usepackage{stmaryrd}
\usepackage{bbm}
\usepackage{mathrsfs}
\usepackage{enumerate}
\usepackage{latexsym,amsxtra}
\usepackage{graphicx}
\usepackage{dsfont}
\usepackage{slashed}
\usepackage[all]{xy}
\usepackage{amscd,graphics}
\usepackage{amsmath,amsfonts,amsthm,amssymb}
\usepackage{latexsym,amsmath}
\usepackage{graphicx,psfrag}
\usepackage{tikz-cd} 
\usepackage{mathabx}

\usepackage{fancyhdr}

\usepackage{hyperref}

\numberwithin{equation}{section}

\newtheorem{thm}{Theorem}[section]
\newtheorem{introthm}{Theorem}

\newtheorem{assum}[thm]{Assumption}

\newtheorem{lem}[thm]{Lemma}

\theoremstyle{definition}
\newtheorem{defi}[thm]{Definition}

\theoremstyle{remark}
\newtheorem{rmk}{Remark}

\newcommand{\spinc}{\textnormal{spin}^{c}}
\newcommand{\Met}{\textnormal{Met}}

\newcommand{\FSW}{\textnormal{FSW}}
\newcommand{\PD}{\textnormal{PD}}

\newcommand{\Diff}{\textnormal{Diff}}
\newcommand{\Symp}{\textnormal{Symp}}
\newcommand{\Conf}{\textnormal{Conf}}

\newcommand{\Ker}{\textnormal{Ker}}

\newcommand{\Image}{\textnormal{Im}}
\newcommand{\p}{\textnormal{p}}

\newcommand{\End}{\textnormal{End}}
\newcommand{\U}{\textnormal{U}}

\newcommand{\T}{\textnormal{T}}
\newcommand{\Fr}{\textnormal{Fr}}
\newcommand{\WA}{\textnormal{WA}}
\newcommand{\KS}{\textnormal{KS}}

\newcommand{\SO}{\textnormal{SO}}

\title{Family Seiberg-Witten equation on K\"ahler surface and $\pi_i(\Symp)$ on multiple-point blow ups of Calabi-Yau surfaces}
\author{Yi Du}
\date{}
\address{ School of Mathematics, University of Minnesota, Minneapolis, MN, US }

\email{du000187@umn.edu}

\begin{document}
\begin{abstract}
Let $\omega$ be a K\"ahler form on $X$, which is a torus $T^4$ or a $K3$ surface such that $\kappa=[\omega]$ satisfies certain irrationality condition. Let $(\tilde X,\tilde\omega)$ be an $n-$point K\"ahler blowup on $(X,\omega)$. We extend the techniques developed by Kronheimer\cite{Kronheimer1998} and Smirnov\cite{Smirnov2022}\cite{Smirnov2023} and show that even dimensional higher homotopy groups of (a subgroup of) $\Symp(\tilde X,\tilde\omega)$ are infinitely generated.
\end{abstract}
\maketitle

\section{Introduction}\label{section: introduction}
Let $(X,\omega)$ be a connected closed symplectic $4$-manifold, $\Diff_0(X)$ be the group of diffeomorphisms isotopic to the identity, and $\Symp(X,\omega)$ be its symplectomorphism group. The topology of symplectomorphism groups has been studied extensively for rational or ruled surfaces \cite{Gromov1985}\cite{AbreuMcduff2000}\cite{LalondePinsonnault2002}\cite{AnjosLiLiPinsonnault2019}\cite{Mcduff2000}\cite{LiLiWu2022Torelli}. For closed symplectic manifolds with $b^+>1$, there has also been substantial progress on symplectic mapping class groups \cite{SheridanSmith2020}\cite{Smirnov2022}\cite{Smirnov2023}\cite{MunozEchanez25}\cite{KonnoLinMukherjeeMunoz-Echaniz2025}\cite{DuLi2026I}. In this paper, we focus on the higher homotopy groups of symplectomorphism groups of certain symplectic $4$-manifolds with $b^+>1$.\\
Let $\mathcal{S}_{\omega}$ be the space of symplectic forms isotopic to $\omega$, and denote $$\Symp_{st}(X,\omega):=\Symp(X,\omega)\bigcap\Diff_0(X).$$
Kronheimer's fibration 
\begin{equation}\label{equa: Kro fibration}
\Symp_{st}(X,\omega)\xrightarrow{\iota}\Diff_0(X)\xrightarrow{\p}\mathcal{S}_{\omega}
\end{equation}
and the corresponding long exact sequence relate $K(X,\omega):=\pi_0(\Symp_{st}(X,\omega))$ and higher homotopy groups of $\Symp_{st}(X,\omega)$ to homotopy groups of $\mathcal S_\omega$ .
In \cite{Kronheimer1998}, Kronheimer uses the family Seiberg-Witten theory to show that the homotopy classes of $\mathcal{S}_\omega$ can be non-trivial. Smirnov\cite{Smirnov2022}\cite{Smirnov2023} refines this construction and proves that $K(X,\omega)$ is infinitely generated for certain K\"ahler forms $\omega$ on $K3$ surfaces and one-point blowups of complex two-tori. Further study of $K(X,\omega)$ is discussed in \cite{MunozEchanez25}\cite{DuLi2026I}, and related applications to the study of $\pi_1(\Diff_0(X),id)$ appear in \cite{Smirnov2020}\cite{Lin2022}. In this paper, we extend these methods to higher homotopy groups of symplectomorphism groups of multiple-point blowups of Calabi-Yau surfaces.

Let $(X,\omega)$ be a K\"ahler surface, where $X$ is a $K3$ surface or a complex two-torus, and suppose that $\kappa=[\omega]$ is non-resonant, i.e. $\langle\kappa,\delta\rangle\neq0$ for every non-zero $\delta\in H^2(X;\mathbb{Z})$, where $\langle\ ,\ \rangle$ is the intersection pairing.
\begin{introthm}\label{thm: infinitely generated symp family}
Let $(\tilde X,\tilde\omega)$ be an $n$-point K\"ahler blowup of $(X,\omega)$, let $e_1,\cdots,e_n$ be the classes of exceptional divisors. Suppose that the areas of exceptional divisors are equal to $\lambda>0$, satisfying $2n\lambda<\sqrt{\kappa^2}.$ Then 
$$\Ker(\iota_{j,*}:\pi_j(\Symp_{st}(\tilde X,\tilde\omega))\rightarrow\pi_j(\Diff_0(\tilde X)))$$
is infinitely generated for
$$j=2,4,\cdots,2n-2\quad\text{when }X=T^4,\qquad
j=2,4,\cdots,2n\quad\text{when }X\text{ is a }K3\text{ surface},$$
where $\iota_{j,*}$ is induced by inclusion $\iota$.
\end{introthm}
Now we sketch the idea of proof. Set $K=\sum_{i=1}^ne_i$, which is the canonical class of $\tilde\omega$. Let $\Delta_{X,k}$ be the infinite set of
integral classes $\delta$ satisfying
$$(k-1)\lambda<\langle\delta,\kappa\rangle<k\lambda,$$
where $\delta$ is primitive with $\delta^2=0$ for tori,
and $\delta^2=-2$ for $K3$ surfaces. Set  $\epsilon_\delta=\delta-\sum_{i=1}^k e_i$, we extend Smirnov's $q$-invariants to homomorphisms
$$q_{\epsilon_\delta}
=Q_{\epsilon_\delta}-Q_{K-\epsilon_\delta}:
\pi_{2r-1}(\mathcal S_{\tilde\omega})
\longrightarrow\mathbb Z_2,$$
where $r=k$ for $T^4$, $r=k+1$ for K3 surface, $Q$-invariants are the Kronheimer's invariants. Let
$$\cdots\xrightarrow{\iota_{2r-1,*}}\pi_{2r-1}(\Diff_0(\tilde X))\xrightarrow{\p_{2r-1,*}}\pi_{2r-1}(\mathcal{S}_{\tilde\omega})\xrightarrow{\partial_{2r-1,*}}\pi_{2r-2}(\Symp_{st}(\tilde X,\tilde\omega))\rightarrow\cdots$$
be the long exact sequence corresponding to Kronheimer's fibration \eqref{equa: Kro fibration}.
For each $r\ge 2$, $q_{\epsilon_\delta}$ vanishes on the image of $\p_{2r-1*}$.
For every finite subset $F\subset\Delta_{X,k}$, we shall prove that
\[
(q_{\epsilon_\delta})_{\delta\in F}:
\pi_{2r-1}(\mathcal S_{\tilde\omega})
\longrightarrow\bigoplus_{\delta\in F}\mathbb Z_2
\]
is surjective. These maps descend to $\Ker(\iota_{2r-2,*})$,
which is therefore not finitely generated.

We briefly explain the construction of the spherical families used to prove these surjectivity statements. 

Following \cite{Smirnov2022}\cite{Smirnov2023}, we begin with
a disk $D_\delta$ of complex structures of $X$ on which $\delta$ is represented
by a holomorphic curve only in the central fiber. When $X$ is a K3 surface, we consider the relative
configuration space of $n$ ordered blowup centres over $D_\delta$, and the
locus $\mathcal H_\delta$ where the first $k$ centers lie on such a curve. When $X=T^4$, we consider corresponding spaces where the center of the exceptional divisor in $e_1$ is fixed by a specific section. The corresponding locus has complex codimension
$$r=k\quad\text{for tori},\qquad
r=k+1\quad\text{for }K3\text{ surfaces}.$$
Let $B_\delta\cong D^{2r}$ be a small disk family intersecting
$\mathcal H_\delta$  transversally once, it parametrizes a family $\{\tilde\omega_q|q\in B_\delta\}$ of K\"ahler forms such that
$$[\tilde\omega_q]=\kappa-\lambda\sum_{i=1}^ne_i,\quad \forall q\in S_\delta=\partial B_\delta\cong S^{2r-1}.$$
Then $\{\tilde\omega_q| q\in S_\delta\}$ defines an element $\alpha_\delta\in\pi_{2r-1}(\mathcal S_{\tilde\omega})$.
We show that
$$q_{\epsilon_{\delta'}}(\alpha_\delta)=
\begin{cases}
1,&\delta'=\delta,\\
0,&\delta'\ne\delta.
\end{cases}$$

\textbf{AI use}. The core ideas and first version of this paper were developed two years ago, before the emergence of ChatGPT 5. In the latest revision, the author uses ChatGPT to prepare the manuscript. 

In \textbf{Section 2}, we introduce preliminaries about K\"ahler geometry and period domains. In \textbf{Section 3}, we generalize the family Seiberg-Witten invariants and prove a regularity theorem. In \textbf{Section 4}, we construct the spherical families of symplectic forms and prove the main theorem.\\
$\textbf{Acknowledgement}$. The author would like to greatly thank Jianfeng Lin for the suggestion to study the topology of symplectomorphism group and for a lot of helpful discussion during the research. I'm also very grateful to my advisor, Tianjun Li, for a lot of helpful advice about the paper and help in symplectic geometry and gauge theory. I'd like to thank Gleb Smirnov for kindly explaining his paper to me. I would like to thank Alexander Voronov and Weiwei Wu for helpful discussion about complex geometry. And I would like to thank Hokuto Konno and Jun Li for their interest in this paper. Also, I'd like to thank the anonymous referee for valuable advice.

\section{K\"ahler classes, period domains on $X$, and parameter spaces of blowing up on $\tilde X$}
\label{sec: Preliminary knowledge in Kahler geometry and symplectic geometry}

Throughout this section, $X$ is a $K3$ surface or a complex torus $T^4$. We identify curve classes with their Poincar\'e duals when taking
intersection products.

\subsection{The K\"ahler cone and point blowups}
Write $\mathcal K(Y)\subset H^{1,1}(Y;\mathbb R)$ for the K\"ahler
cone of a compact K\"ahler surface $Y$, and $V^+_Y$ for the component of $\{a\in H^{1,1}(Y;\mathbb R)|a^2>0\}$ containing a K\"ahler class. By \cite[Theorem~0.1]{DemaillyPaun2004}, we have
\begin{equation}\label{eq:2-kahler-criterion}
\mathcal K(Y)=\{a\in V^+_Y|a\cdot D>0
\text{ for every reduced irreducible curve }D\subset Y\}.
\end{equation}
Let $\pi:\widetilde X\to X$ be the blowup at distinct points
$x_1,\ldots,x_n$. Write $E_i$ for the corresponding exceptional curves and set
$e_i=\operatorname{PD}[E_i]$, so that $e_i^2=-1$.
For $\kappa\in\mathcal K(X)$ and $\boldsymbol\lambda=(\lambda_1,...,\lambda_n)$ with $\lambda_i>0$, put $\alpha_{\boldsymbol\lambda}=\pi^*\kappa-\sum_{i=1}^n\lambda_i e_i$. If $C\subset X$ is reduced and irreducible, set
$m_i=\operatorname{mult}_{x_i}C$, with $m_i=0$ when $x_i\notin C$.
Its proper transform satisfies
\begin{equation}\label{eq:2-strict-transform-area}
\PD[\widetilde C]=\pi^*\PD[C]-\sum_i m_i e_i,
\qquad
\alpha_{\boldsymbol\lambda}\cdot\widetilde C
=\kappa\cdot C-\sum_i\lambda_i m_i.
\end{equation}
\begin{lem}\label{lem: bound of blowup kahler class}
The class $\alpha_{\boldsymbol\lambda}$ is K\"ahler if and only if
\begin{equation}\label{eq:41-positivity}
\kappa^2>\sum_i\lambda_i^2,
\qquad
\kappa\cdot C>\sum_i\lambda_i\operatorname{mult}_{x_i}C
\quad\text{for every reduced irreducible }C\subset X.
\end{equation}
\end{lem}
\begin{proof}
Since $\alpha_{\boldsymbol\lambda}^2
=\kappa^2-\sum_i\lambda_i^2$, If $\alpha_{\boldsymbol{\lambda}}$ is a K\"ahler class, then $\alpha_{\boldsymbol{\lambda}}^2=\kappa^2-\sum_i\lambda_i^2>0$ and $\alpha_{\boldsymbol\lambda}\cdot \widetilde C=\kappa\cdot C-\sum_i\lambda_i\operatorname{mult}_{x_i}C>0
$ for the proper transform of every reduced irreducible curve $C\subset X$. Then \eqref{eq:41-positivity} follows.

Now suppose that $\alpha_{\boldsymbol{\lambda}}$ satisfies \eqref{eq:41-positivity}. Let $\boldsymbol{\lambda}^\prime=(\lambda^\prime_1,...,\lambda^\prime_n)$ such that $\lambda^\prime_i$ is sufficiently small for each $i$, we can see that $\alpha_{\boldsymbol{\lambda}^\prime}$ is a K\"ahler class. 
For any $\boldsymbol{\lambda}=(\lambda_1,\dots,\lambda_n)$, we have
$$\alpha_{\boldsymbol\lambda}\cdot \alpha_{\boldsymbol\lambda^\prime}=\kappa^2-\sum_i\lambda_i\lambda_i^\prime>\kappa^2-\sum_i\lambda_i^2>0,$$
thus $\alpha_{\boldsymbol{\lambda}}$ is always contained in $V^+_{\tilde X}$. 

If $D\subset\widetilde{X}$ is an irreducible curve, then $D=E_i$ or $D$ is the proper transform of some curve $C$
on $X$. $\alpha_{\boldsymbol{\lambda}}\cdot E_i=\lambda_i>0$. Then the Lemma follows from [Theorem $0.1$, \cite{DemaillyPaun2004}].
\end{proof}

\begin{lem}\label{lem: no bound from positive curves and small blowup}
Let $\pi:\widetilde X\to X$ be the blowup at distinct points
$x_1,\ldots,x_n$, with exceptional divisors $E_1,\ldots,E_n$.
Write $e_i=\operatorname{PD}[E_i]$ and set
\begin{equation}
\alpha_{\boldsymbol{\lambda}}
=\pi^*\kappa-\sum_{i=1}^n\lambda_i e_i,
\qquad
S=\sum_{i=1}^n\lambda_i,
\qquad
\ell=\Bigl(\sum_{i=1}^n\lambda_i^2\Bigr)^{1/2},
\end{equation}
where $\lambda_i>0$ and $2S<\sqrt{\kappa^2}$. Then 
\begin{equation}\label{eq:41-volume}
\alpha_{\boldsymbol{\lambda}}\in V^+_{\widetilde X}, \qquad
\alpha_{\boldsymbol{\lambda}}\cdot E_i=\lambda_i>0.
\end{equation}
If $\alpha_{\boldsymbol{\lambda}}\notin \mathcal{K}(\widetilde{X})$, or equivalently, there is a proper transform $\widetilde C$ of a reduced
irreducible curve $C\subset X$ such that
$\alpha_{\boldsymbol{\lambda}}\cdot[\widetilde C]\le0$, then
\begin{equation}\label{equa : bound of area for negative curve classes}
-c_X\le C^2\le0,
\qquad
0<\kappa\cdot[C]\le S+\ell\sqrt{c_X},
\qquad
c_X=
\begin{cases}
2,&X\text{ is a K3 surface},\\
0,&X\text{ is a complex two-torus}.
\end{cases}
\end{equation}
\end{lem}

\begin{proof}
Since $K_X=0$, the adjunction formula and the genus formula for a
point blowup imply that
\[
p_a(C)=1+\frac{C^2}{2},
\qquad
p_a(\widetilde C)
=1+\frac{C^2}{2}
-\frac12\sum_{i=1}^n m_i(m_i-1),
\]
where $m_i=\operatorname{mult}_{x_i}C$.
Since $p_a(\widetilde C)\ge 0$ and it's bigger than zero when $X$ is $T^4$, we have
\[
C^2\ge-c_X,
\qquad
\sum_{i=1}^n m_i(m_i-1)\le C^2+c_X.
\]
Suppose that $C^2>0$. Since $C^2$ is even, we have $C^2\ge2$.
For $m_i\ge2$,
\[
m_i^2\le2m_i(m_i-1)
\le2(C^2+c_X)\le4C^2.
\]
Thus $m_i\le2\sqrt{C^2}$ for every $i$, which also holds when $m_i=0$ or $1$.
The Hodge index theorem implies that
\[
\alpha_{\boldsymbol{\lambda}}\cdot[\widetilde C]
=\kappa\cdot[C]-\sum_{i=1}^n\lambda_i m_i
\ge\bigl(\sqrt{\kappa^2}-2S\bigr)\sqrt{C^2}>0,
\]
contrary to the hypothesis. Hence $C^2\le0$.

Using $m_i\le1+\sqrt{m_i(m_i-1)}$ and
Cauchy--Schwarz, we get
\[
\kappa\cdot[C]\le\sum_{i=1}^n\lambda_i m_i
\le S+\ell
\Bigl(\sum_{i=1}^n m_i(m_i-1)\Bigr)^{1/2}
\le S+\ell\sqrt{C^2+c_X}
\le S+\ell\sqrt{c_X}.
\]
This proves \eqref{equa : bound of area for negative curve classes}.
Finally, \eqref{eq:41-volume} follows from the proof of Lemma \ref{lem: bound of blowup kahler class}.
\end{proof}

\subsection{Period domains of K3 surfaces and complex two-tori}
\label{subsec:2.2}
Following the method developed by Smirnov\cite{Smirnov2022}\cite{Smirnov2023}, we use period domains to get parameter spaces of the complex structures of a $K3$ surface or a complex two-torus. 

The following lattices are isomorphic to the second cohomology groups of $X$:
\[
\Lambda=
\begin{cases}
3H\oplus2(-E_8),&\text{for K3 surfaces},\\
\bigwedge^2\Gamma\cong3H,&\text{for complex two-tori},
\end{cases}
\]
where $H=\begin{pmatrix}
0 & 1 \\
1 & 0
\end{pmatrix}$ and $\Gamma$ is an oriented
free abelian group of rank four. The orientation identifies
$\bigwedge^4\Gamma$ with $\mathbb Z$, so exterior product defines
the intersection pairing on $\bigwedge^2\Gamma$.
Write $\Lambda_{\mathbb F}=\Lambda\otimes_{\mathbb Z}\mathbb F$
for $\mathbb F=\mathbb R,\mathbb C$. For a torus, we choose an orientation-preserving marking
$\eta_X:H^1(X;\mathbb Z)\to\Gamma$ and retain the induced marking
\[
\mu_X=\bigwedge^2\eta_X:
H^2(X;\mathbb Z)=\bigwedge^2H^1(X;\mathbb Z)
\longrightarrow\Lambda.
\]
For a K3 surface, a marking is an isometry
$\mu_X:H^2(X;\mathbb Z)\to\Lambda$.

For either surface, $H^{2,0}(X,\mathbb{C})$ is a complex line.  Its image
\[
\operatorname{Per}(X,\mu_X)
:=\mu_{X}\bigl(H^{2,0}(X)\bigr)
\in\mathbb P(\Lambda_{\mathbb C})
\]
is the \emph{period point} of the marked surface.
\emph{The period domain} of $X$ is defined as follows.
\begin{equation}\label{equa: definition of period domain of K3 surface}
\mathcal M_X=
\bigl\{[\sigma]\in\mathbb P(\Lambda_{\mathbb C}):
\sigma\cdot\sigma=0,\quad
\sigma\cdot\overline\sigma>0\bigr\}.
\end{equation}
We write $u=[\sigma]$ for a period point and
$\ell_u=\mathbb C\sigma$, and abbreviate $\mathcal M_X$ to
$\mathcal M$ when the surface type is understood.
In a holomorphic family $\mathcal{X}\to B$ equipped with locally constant markings,
these period points define a holomorphic map from $B$ to
$\mathcal M$, called the \emph{period map}.

Let $(X,\mu_X)$ be a marked $K3$ surface or $T^4$ with period $u_0\in\mathcal{M}_X$. For a K3 surface, the semi-universal deformation of $X$ has
a smooth base of complex dimension $20$, and period map of the deformation family is a local biholomorphism by
[Chapter~6, Corollary~2.7 and Proposition~2.8, \cite{Huybrechts2016}]. Thus its base can be identified with a sufficiently small neighbourhood $U$ of $u_0$ in $\mathcal M$.
For tori, the marked family and the period isomorphism in
[Section~7, Theorems~7--8, \cite{Smirnov2023}], following
\cite{Shioda1978}, give the same conclusion, with a four-dimensional base together with a holomorphic zero section. In either case, there is a proper holomorphic submersion
\begin{equation}\label{eq:22-local-period-family}
p:\mathcal X_U\longrightarrow U,
\qquad X_u:=p^{-1}(u),
\qquad \operatorname{Per}(X_u,\mu_u)=u
\end{equation}
for a sufficiently small neighborhood $U$ of $u_0$.

Fix $\kappa\in\Lambda_{\mathbb R}$ with $\kappa^2>0$.
To keep this class of type $(1,1)$, we restrict the period to
\begin{equation}\label{equa: defnition of polarized domain of K3 surface}
\overline{\mathcal M}_\kappa
=\{u\in\mathcal M:\kappa\perp\ell_u\}.
\end{equation}
The pairing on $\kappa^\perp$ has signature $(2,19)$ for K3
surfaces and $(2,3)$ for tori.  Consequently,
$\overline{\mathcal M}_\kappa$ has two contractible connected
components; see \cite[Section~8]{Smirnov2022} and
\cite[Section~7]{Smirnov2023}.
Let $(X,\mu_X)$ be a marked surface such that $\mu_X^{-1}(\kappa)\in\mathcal K(X)$, and denote the component containing its period by
$\overline{\mathcal M}_\kappa^{+}$.
For $a\in\Lambda$, write
\begin{equation}
H_a:=\{u\in\mathcal M:a\perp\ell_u\}.
\end{equation}
\begin{lem}
Set $S(X,\kappa)$ to be $\{a\in\Lambda|a^2=-2,\langle \kappa,a\rangle=0\}$, then the $\kappa$-polarized period domain 
\begin{equation}\label{eq:22-polarized-period-locus}
\mathcal M_\kappa=
\begin{cases}
\displaystyle
\overline{\mathcal M}_\kappa^{+}\setminus
\bigcup_{\substack{a\in S(X,\kappa)}}H_a,
&\text{for K3 surfaces},\\[6pt]
\overline{\mathcal M}_\kappa^{+},
&\text{for complex two-tori}.
\end{cases}
\end{equation}
is a fine moduli space of marked surfaces for which
$\kappa$ is K\"ahler, in the sense that
\[
p_\kappa:\mathcal X_\kappa\longrightarrow\mathcal M_\kappa
\]
is a holomorphic submersion, and there is a smooth family $\{\omega_u|u\in\mathcal{M}_\kappa\}$ of K\"ahler forms on $X_u$ such that $[\omega_u]=\kappa$. When $X=T^4$, there is a zero section $Z:\mathcal{M}_\kappa\rightarrow \mathcal{X}_\kappa$. 
\end{lem}
\begin{proof}
See \cite[Section~7]{Smirnov2022}
\cite[Section~7]{Smirnov2023}.    
\end{proof}
By the definition of non-resonance, we have the following result.
\begin{lem}
When $X$ is a K3 surface and $\kappa$ is non-resonant, $\mathcal{M}_\kappa=\overline{\mathcal{M}}^+_\kappa$.
\end{lem}
\begin{lem}\label{lem: contractibility of polarized period domains}
When $\kappa$ is non-resonant, $\mathcal{M}_\kappa=\overline{\mathcal{M}}^+_\kappa$ is contractible.
\end{lem}
\begin{proof}
The contractibility of $\overline{\mathcal{M}}^+_\kappa$ follows from \cite{BarthPetersVandevenHulek2015}\cite{Smirnov2022}. See also [Lemma $3.6$, \cite{DuLi2026I}].
\end{proof}
Now we introduce some sub-domains of $\mathcal{M}_\kappa$, related to parameter space of multiple-point blowup of $X$, with prescribed blowup sizes.

Set
\[
\operatorname{Conf}_n(X_u)
=\{(x_1,\ldots,x_n)\in X_u^n|x_i\ne x_j\text{ if }i\ne j\}.
\]
For $\mathbf x\in\operatorname{Conf}_n(X_u)$, let
$\pi_{u,\mathbf x}:\widetilde X_{u,\mathbf x}\to X_u$
be the natural projection, and write
$e_i=\operatorname{PD}[E_i]$ for the exceptional classes.
With the identification $\mu_u: H^2(X_u;\mathbb{Z})\cong \Lambda$, $\kappa$ can be considered as a class on $X_u$.
\begin{defi}\label{defi: M_kappa,lambda}
Let $\boldsymbol\lambda=(\lambda_1,\ldots,\lambda_n)$ with
$\lambda_i>0$ and $\sum_i\lambda_i^2<\kappa^2$.
For distinct points $\mathbf{x}=(x_1,\ldots,x_n)\in \Conf_n(X_u)$, let
$\tilde X_{u,x_1,\ldots,x_n}\to X_u$
be the blowdown, let $e_i$ be the corresponding
exceptional classes and set $\alpha_{u,\mathbf x}
=\pi_{u,\mathbf x}^{*}\kappa-\sum_{i=1}^n\lambda_i e_i$. Define
\begin{equation}\label{eq:22-universal-period-locus}
\begin{split}
\mathcal M_{\kappa,\boldsymbol\lambda}
=\bigl\{u\in\mathcal M_\kappa:\;&
\pi_{u,\mathbf x}^*\kappa-\sum_i\lambda_i e_i
\in\mathcal K(\operatorname{Bl}_{x_1,\ldots,x_n}X_u),\ \forall
\mathbf{x}\in \Conf_n(X_u)\bigr\}.
\end{split}
\end{equation}
For a single blowup of size $S$, write $\mathcal M_{\kappa,S}$.
\end{defi}

By Lemma~\ref{lem: bound of blowup kahler class}, then $\alpha_{u,\mathbf x}\in\mathcal K(\widetilde X_{u,\mathbf x})$ is equivalent to
\[
\kappa\cdot[C]>
\sum_{i=1}^n\lambda_i\operatorname{mult}_{x_i}C
\text{ for every reduced irreducible curve }C\subset X_u.
\]
\begin{lem}\label{lem: inclusion of blowup domains}
$\mathcal M_{\kappa,\sum \lambda_i}
\subseteq\mathcal M_{\kappa,\lambda_1,\ldots,\lambda_n}$.
\end{lem}
\begin{proof}
$u\in\mathcal M_{\kappa,\sum_i\lambda_i}$ implies that $\kappa\cdot[C]>
(\sum_{i=1}^n\lambda_i)\operatorname{mult}_xC $ for every reduced irreducible curve $C\subset X_u$. Let $x=x_i$ for each $i$, and suppose that $\operatorname{mult}_{x_1}C\ge \operatorname{mult}_{x_i}C$ for each $i$. Then we have $$\kappa\cdot[C]>(\sum_{i=1}^n\lambda_i)\operatorname{mult}_{x_1}C\ge
\sum_{i=1}^n\lambda_i\operatorname{mult}_{x_i}C$$
for each irreducible $C$ in $X_u$.
\end{proof}
\begin{lem}\label{lem: connectivity of blowup-period domain}
Suppose that $S:=\sum_{i=1}^n\lambda_i<\tfrac12\sqrt{\kappa^2}$, then $\mathcal M_{\kappa,\lambda_1,\ldots,\lambda_n}$ is nonempty, dense, and path connected.
\end{lem}

\begin{proof}
To exclude the possible obstructions from curves of
nonpositive square, set
\[
\mathcal I_X(S)=
\begin{cases}
\{a\in\Lambda|a^2\in\{-2,0\},\ 0<\kappa\cdot a\le2S\},
&X\text{ a K3 surface},\\
\{a\in\Lambda|a^2=0,\ 0<\kappa\cdot a\le S\},
&X\text{ a torus},
\end{cases}
\]
and define
\[
\mathcal G_{\kappa,S}
=\mathcal M_\kappa\setminus
\bigcup_{a\in\mathcal I_X(S)}H_a.
\]
For each $a\in\mathcal I_X(S)$,
$H^\kappa_a:=H_a\bigcap\mathcal M_\kappa$ is a smooth complex hypersurface. By [Lemma $12$, \cite{Smirnov2023}], hyperplanes in $\bigcup_{a\in \mathcal I_X(S)} H^\kappa_a$ are locally finite. By [Lemma~5,\cite{Smirnov2022}], $\bigcup_{\{a\in\mathcal I_X(S)|a^2=-2\}} H^\kappa_a$ is locally finite. And a similar result holds for $\bigcup_{\{a\in\mathcal I_X(S)|a^2=0\}}H^\kappa_a$ by [Lemma $12$, \cite{Smirnov2023}], see also \cite{DuLi2026I}. As a result, $\mathcal G_{\kappa,S}$ is open and dense in $\mathcal{M}_\kappa$. Since $\mathcal M_\kappa$ is path
connected, the complement of countably many complex hyperplanes in $\mathcal{M}_\kappa$ is path connected.


Let $u\in \mathcal{G}_{\kappa,S}$ and let $C\subset X_u$ be an irreducible curve s.t. $C^2\le0$. 

By the adjunction formula on $K3$ surface, $C^2\in\{-2,0\}$. Let $\widetilde C$ be the proper transform of $C$ under the
blowup at $x$. Let $m$ be the multiplicity if $x\in C$, and set $m=0$ if $x\notin C$. Then
\[
p_a(\widetilde C)
=1+\frac{C^2}{2}-\frac{m(m-1)}{2}\ge 0,
\]
thus $m(m-1)\le C^2+2\le2$, $m\le2$. By the definition of $\mathcal{I}_X(S)$, we have
\[
\kappa\cdot C>2S\ge S\cdot m.
\]
For a complex two-torus, an irreducible curve $C$ with $C^2\le 0$ is a smooth elliptic curve of square zero. Hence $m\le1$, while the definition of $\mathcal I_{X}(S)$
implies
\[
\kappa\cdot C>S\ge Sm.
\]

Now let $C\subset X_u$ be irreducible curve with $C^2>0$. By Lemma~\ref{lem: no bound from positive curves and small blowup}, $\kappa\cdot[C]>S\operatorname{mult}_xC$ for every reduced irreducible curve $C\subset X_u$
and every $x\in X_u$. In conclusion, if $u\in\mathcal{G}_{\kappa,S}$, then $\pi_{u,x}^*\kappa-Se$ is a K\"ahler class on $\tilde X_{u,x}$ for each $x\in X_u$, where $e$ is the exceptional class of $\tilde X_{u,x}$ for each $x \in X_u$, thus $\mathcal G_{\kappa,S}\subseteq\mathcal M_{\kappa,S}$. By Lemma \ref{lem: inclusion of blowup domains}, we have
\[
\mathcal G_{\kappa,S}
\subseteq\mathcal M_{\kappa,S}
\subseteq\mathcal M_{\kappa,\lambda_1,\ldots,\lambda_n}.
\]


Finally, let $u\in\mathcal M_{\kappa,\lambda_1,\ldots,\lambda_n}$.
By local finiteness of hyperplanes parametrized by $\mathcal{I}_X(S)$, we can choose a path
$\gamma:[0,1]\to\mathcal M_\kappa$ such that
\[
\gamma(0)=u,\qquad
\gamma(t)\in\mathcal G_{\kappa,S}\quad(0<t\le1).
\]
Since $u$ and $\mathcal G_{\kappa,S}$ both belong to
$\mathcal M_{\kappa,\lambda_1,\ldots,\lambda_n}$,
the whole path lies in this locus. As
$\mathcal G_{\kappa,S}$ is dense and path connected,
the claimed nonemptiness, density, and path connectedness follow.

\end{proof}

\subsection{Local parameter families for multiple-point blowups}
\label{subsec: n-point blowups of K3 and T^4}
In this section, we construct some parameter families of K\"ahler structures on the blowup of $X$, under the non-resonance assumtion of $\kappa$.

By Lemma \ref{lem: contractibility of polarized period domains}, the holomorphic family
\[
p_\kappa:\mathcal{X}_\kappa \rightarrow \mathcal{M}_\kappa
\]
is smoothly trivial. From now on, we fix a specific smooth trivialization of $\mathcal{X}_\kappa$.
Thus $\kappa$ is a K\"ahler class on every
fiber $X_u$. When $X=T^4$, there is a natural holomorphic section obtained from the global section $Z$. Here, $\dim_{\mathbb C}\mathcal{M}_\kappa=19$ for K3 surfaces and $3$ for tori. 
\begin{defi}\label{defi: relative configuration space as parameter spaces}
We define the relative configuration space of $\mathcal{X}_\kappa$ 
\begin{equation}\label{diagram:P bar}
r:\overline{\mathcal P}_\kappa:
=\operatorname{Conf}_n(\mathcal X_\kappa/\mathcal{M}_\kappa)
=\left(\mathcal X_\kappa\mathop{\times}_{\mathcal{M}_\kappa}\cdots
\mathop{\times}_{\mathcal{M}_\kappa}\mathcal X_\kappa\right)
\setminus\bigcup_{i<j}\{x_i=x_j\}\to \mathcal{M}_\kappa,\quad (u,\mathbf{x})\to u
\end{equation} 
then $r^{-1}(u)=\operatorname{Conf}_n(X_u)$. Recall the canonical zero section $Z:\mathcal{M}_\kappa\rightarrow\mathcal{X}_\kappa$ when $X$ is a complex torus, then set
\begin{equation}\label{equa: P}
\mathcal P_{\kappa}=
\begin{cases}
\overline{\mathcal P}_\kappa,&\text{for K3 surfaces},\\
\{(u,x_1,\ldots,x_n)\in\overline{\mathcal P}_\kappa:x_1=Z(u)\},
&\text{for tori}.
\end{cases}
\end{equation}
\end{defi}
\begin{rmk}
$\mathcal P_{\kappa}$ parametrizes configurations modulo translation in the torus case. 
\end{rmk} 
The complex dimensions
of $\overline{\mathcal P}_\kappa$ and $\mathcal P_{\kappa}$ are respectively
$2n+3$ and $2n+1$ for tori, and both are $2n+19$ for K3 surfaces.

With the smooth trivialization
$\mathcal X_\kappa\simeq \mathcal{M}_\kappa\times X_{u_0}$, chosen to send $Z$ to the constant zero section for tori,
we have identifications
\begin{equation}\label{equa: parameter spaces}
\begin{aligned}
\overline{\mathcal P}_\kappa&\simeq
\mathcal{M}_\kappa\times\operatorname{Conf}_n(X_{u_0}),\\
\mathcal P_{\kappa}&\simeq
\mathcal{M}_\kappa\times\operatorname{Conf}_{n-1}(X_{u_0}\setminus\{0\})
&&\text{for tori},
\end{aligned}
\end{equation}
Set $\mathcal Y=\overline{\mathcal P}_\kappa\times_{\mathcal{M}_\kappa}\mathcal X_\kappa
\longrightarrow\overline{\mathcal P}_\kappa$, each marked point $\mathbf{x}=(x_1,...,x_n)\in\Conf_n(X_u)$ defines n tautological holomorphic sections
$s_i(q)=(q,x_i)$. Since images of $s_i$ are disjoint, we may blow up
all of them:
\begin{equation}\label{eq:23-blowup-family}
\Pi:\widetilde{\mathcal X}
=\operatorname{Bl}_{\coprod_i s_i(\overline{\mathcal P}_\kappa)}\mathcal Y
\longrightarrow\mathcal Y,
\qquad
\widetilde p:\widetilde{\mathcal X}\longrightarrow
\overline{\mathcal P}_\kappa.
\end{equation}
$\widetilde p$
is a proper holomorphic submersion with fiber $\widetilde X_q=\operatorname{Bl}_{x_1,\ldots,x_n}X_u$.
Let $\pi_q:\widetilde X_q\to X_u$ be its blowdown and
$e_i=\operatorname{PD}[E_{i,q}]$ the exceptional classes. $\widetilde{\mathcal X}$ is a homological trivial family, we can identify the cohomology
local system with
$\Lambda\oplus\bigoplus_{i=1}^n\mathbb Z e_i$. From now on, we also denote the pullback $\pi_q^*\kappa$
by $\kappa$. Thus $\kappa-\sum_i\lambda_i e_i$ denotes
the class $\pi_q^*\kappa-\sum_i\lambda_i e_i$
on $\widetilde X_q$.
The family over $\mathcal P_{\kappa}$ is the restriction of
\eqref{eq:23-blowup-family}. 

Set $\boldsymbol\lambda=(\lambda_1,\ldots,\lambda_n),\
\lambda_i>0,\ \sum_i\lambda_i^2<\kappa^2$ and $\alpha_{\kappa,\boldsymbol\lambda}
=\kappa-\sum_i\lambda_i e_i$. Then $\alpha_{\kappa,\boldsymbol{\lambda}}$ can be considered as a $(1,1)$-class of each $\tilde{X}_q$, and we define
\begin{equation}\label{eq:41-kahler-locus}
\overline{\mathcal{P}}_{\kappa,\boldsymbol\lambda}
=\{q\in\overline{\mathcal P}_\kappa:
\alpha_{\kappa,\boldsymbol\lambda}\in\mathcal K(\widetilde X_q)\},\qquad 
\mathcal{P}_{\kappa,\boldsymbol{\lambda}}:=\overline{\mathcal{P}}_{\kappa,\boldsymbol{\lambda}}\bigcap \mathcal P_{\kappa}
\end{equation}

Set $S=\sum_{i=1}^n\lambda_i$ and suppose that $2S<\sqrt{\kappa^2}$. By the definition of
$\mathcal M_{\kappa,\boldsymbol\lambda}$
and the proof of Lemma~\ref{lem: connectivity of blowup-period domain},
\[
r^{-1}(\mathcal G_{\kappa,S})
\subseteq
r^{-1}(\mathcal M_{\kappa,\boldsymbol\lambda})
\subseteq
\overline{\mathcal P}_{\kappa,\boldsymbol\lambda}.
\]
Intersecting with $\mathcal P_{\kappa}$ gives the corresponding inclusions in $\mathcal P_{\kappa}$. 
\begin{lem}\label{lem: connected kahler configuration locus}
$\mathcal P_{\kappa,\boldsymbol\lambda}$ is open, dense and path connected in $\mathcal P_{\kappa}$. For any two $q,q'\in \mathcal P_{\kappa,\boldsymbol\lambda}$ and K\"ahler forms $\tilde{\omega}_q,\widetilde \omega_{q'}$ on $\tilde X_q,\tilde X_{q'}$ in class $\alpha_{\kappa,\boldsymbol\lambda}$, there is a diffeomorphism 
$$h:(\tilde X_q,\tilde{\omega}_q)\rightarrow (\tilde X_{q'},\widetilde \omega_{q'})$$
\end{lem}
\begin{proof}
If $q\in \mathcal{P}_{\kappa,\boldsymbol{\lambda}}$, then there is a K\"ahler form $\tilde\omega_q$ on $\tilde X_q$ with $[\tilde\omega_q]=\alpha_{\kappa,\boldsymbol\lambda}$. With a local smooth trivialization of the blowup family, by the openness of tameness, there is a small neighborhood $U$ of $q$ in $\mathcal P_{\kappa}$ such that $\tilde\omega_q$ tames $\tilde X_{q'}$ for each $q'\in U$. For each $q'\in U$, $\kappa\in H^{1,1}(\tilde X_{q'};\mathbb{R})$, thus $\alpha_{\kappa,\boldsymbol\lambda}$ remains of type $(1,1)$. By [Theorem $1.3$, \cite{LiZhang2009}], $[\tilde\omega_q]$ is a K\"ahler class of $\tilde X_{q'}$. As a result, $\mathcal{P}_{\kappa,\boldsymbol{\lambda}}$ is an open subset of $\mathcal P_{\kappa}$. In particular, $\mathcal{P}_{\kappa,\boldsymbol{\lambda}}$ is locally path connected.

Since $\mathcal G_{\kappa,S}$ is dense and path connected in $\mathcal{M}_\kappa$,
$r^{-1}(\mathcal G_{\kappa,S})\cap\mathcal P_{\kappa}$ is also dense and path connected in $\mathcal P_{\kappa}$. Since $r^{-1}(\mathcal G_{\kappa,S})\cap\mathcal P_{\kappa}\subset\mathcal{P}_{\kappa,\boldsymbol{\lambda}}$, $\mathcal{P}_{\kappa,\boldsymbol{\lambda}}$ is dense in $\mathcal P_{\kappa}$. Together with the local path-connectedness of $\mathcal{P}_{\kappa,\boldsymbol{\lambda}}$, we get the path connectedness of $\mathcal{P}_{\kappa,\boldsymbol{\lambda}}$.

Let $q,q'\in\mathcal{P}_{\kappa,\boldsymbol{\lambda}}$ and let $\tilde\omega_q,\tilde\omega_{q'}$ be K\"ahler forms on $\tilde X_q,\tilde X_{q'}$, respectively. Let $I:[0,1]\rightarrow \mathcal{P}_{\kappa,\boldsymbol{\lambda}}$ be a path connecting $q$ and $q'$, take a smooth trivialization of the blowup family along $I$. Then there is a path of cohomologous K\"ahler forms $\{\tilde\omega_t|t\in[0,1]\}$ connecting $\tilde\omega_q$ and $\tilde\omega_{q'}$. Applying Moser's
isotopy to the path, we get a diffeomorphism between the symplectic structures preserving the cohomology.
\end{proof}


\section{Family Seiberg Witten invariant, Kronheimer and Smirnov's invariant}
\label{sec:3}
\subsection{Family Seiberg Witten equation}\label{subsec:3.1}
Let $X$ be a connected, closed, smooth, oriented $4-$dimensional manifold. Let $b_i$ be $i-$th betti number of $X$ and $b^+$ be the dimension of the positive subspace of $H^2(X;\mathbb{R})$.\\
Let $\pi:\mathcal{X}\rightarrow B$ be a smooth fiber bundle, where the fiber of the bundle is diffeomorphic to $X$. Define $\T_{\mathcal{X}/B}$ to be the vertical tangent bundle of $\mathcal{X}$, i.e. a subbundle of $T\mathcal{X}$ defined as $\ker(\pi_*)$. We choose a smooth family of metric \{$g_b$\} on $\T_{\mathcal{X}/B}$, which defines a principal $\SO(4)$ bundle of frames, $Fr \rightarrow \mathcal X$. A family $\spinc$ structure $\mathcal{L}$ is a family of lifting of the frame bundle $\Fr$, to a principal $\spinc(4)$ bundle over $\mathcal X$. Given a family $\spinc$ structure $\mathcal{L}$, we have two natural $\U(2)$ bundles, $S^+, S^-$. The natural homomorphism $\T^*(\mathcal{X}/B)\otimes S^+ \rightarrow S^-$ defines the Clifford homomorphism, and its adjoint defines a canonical homomorphism 
$$\tau: \End(S^+)\rightarrow \Lambda^+\otimes\mathbb{C}$$
Here, $\Lambda^+:=\Lambda^+T^*_{\mathcal{X}/B}$ is the vector bundle of self-dual vertical $2$-forms over
$\mathcal{X}$, defined by the family of metrics $\{g_b\}$.
\\
Let $L:=\det(S^+)$. A unitary connection $A_b$ on $L_b$,
together with the Levi-Civita connection of $g_b$, determines
a covariant derivative $\nabla_{A_b}$ on the spinor bundles.

Let $\mathcal A(L_b)$ be the space of unitary connections on $L_b$
and set
\[
\mathcal C_b:=C^\infty(X_b,S_b^+)\times\mathcal A(L_b).
\]
Following \cite{LiLiu2001}, define the 
configuration space to be
\begin{equation}\label{defi:configuration space of Seiberg Witten equation}
\mathcal C
:=\{(\Phi,A)=\{(\Phi_b,A_b)\}_{b\in B}
\mid (\Phi_b,A_b)\in\mathcal C_b
\text{ depends smoothly on }b\}.
\end{equation}
Its irreducible part $\mathcal C^*$ consists of families with
$\Phi_b\not\equiv0$ on $X_b$ for every $b\in B$.

Let $\{\eta_b\in\Omega^{2,+}_{g_b}(X_b;\mathbb R)\}$ be a smooth family of real self-dual forms, the perturbed family Seiberg--Witten equations are
\begin{equation}\label{eq: Seiberg Witten equation}
\left\{
\begin{aligned}
\mathcal D_{A_b}\Phi_b&=0,\\
F^+_{A_b}&=\frac14\tau(\Phi_b\otimes\Phi_b^*)+i\eta_b.
\end{aligned}
\right.
\end{equation}
Here
$\mathcal D_{A_b}:C^\infty(X_b,S_b^+)\to C^\infty(X_b,S_b^-)$
is the Dirac operator obtained by composing $\nabla_{A_b}$
with Clifford multiplication. A solution on $X_b$ is reducible
if $\Phi_b\equiv0$, and irreducible otherwise.

By \cite{Kronheimer1998},\cite{LiLiu2001}, the formal dimension of Seiberg-Witten moduli space as a real manifold is: 
\begin{equation}\label{equa: formal dimension of Seiberg Witten moduli space}
d(\mathcal{L},B)
=\frac{c_1(L_b)^2-2\chi(X)-3\sigma(X)}{4}
+\dim_{\mathbb R}B,
\end{equation}
where $L_b=\det(S_b^+)$.

Let $\Gamma_{\mathcal X}$ be the space of smooth families of metrics and perturbations
\begin{equation}\label{equa: parameter space of Seiberg Witten equation}
\Gamma_{\mathcal X}
:=\{(g,\eta)\mid
g=\{g_b|b\in B\},\
\eta=\{\eta_b|b\in B\},\
\eta_b\in\Omega^{2,+}_{g_b}(X_b;\mathbb R)\}.
\end{equation} 
And the moduli space of family Seiberg Witten equation solutions is defined as
\begin{equation}\label{equa: Moduli space of solutions of FSW}
\mathcal{M}^\mathcal{L}_{g,\eta}:=\bigcup_{b\in B}\mathcal{M}^{\mathcal{L}_b}_{g_b,\eta_b}:=\bigcup_{b\in B}\{\text{ Solutions of the family Seiberg-Witten equation at $b$}\}/\mathcal{G}_b
\end{equation}
Here, $\mathcal{G}_b=Map(\mathcal{X}_b,U(1))$, and they form a bundle $\mathcal{G}$ of groups over $B$.\\
Let $\mathcal H_{g_b}$ be the space of $g_b$-harmonic $2$-forms and identify it 
with $H^2(X_b;\mathbb R)$, we use $H_{g_b}^+$ to denote the corresponding
harmonic projection.
By \eqref{eq: Seiberg Witten equation}, a
reducible solution occurs over $b$ if and only if
\[
H_{g_b}^+(\eta_b)
+2\pi H_{g_b}^+(c_1(L_b))=0.
\]
Thus we define the $\mathcal L$-wall to be
\begin{equation}\label{equa:wall of SW parameter space}
\WA(\mathcal L)
:=\left\{(g,\eta)\in\Gamma_{\mathcal X}:
H_{g_b}^+(\eta_b)
+2\pi H_{g_b}^+(c_1(L_b))=0
\text{ for some }b\in B\right\}.
\end{equation}
\begin{defi}\label{defi: Gamma^*}
We define $\Gamma^*_{\mathcal{L}}$ to be $\Gamma_{\mathcal{X}}\backslash \WA(\mathcal{L})$ and define a chamber of $\Gamma^*_{\mathcal{L}}$ to be a connected component of it.
\end{defi}
With the chamber of the parameter family specified, we can define the family Seiberg-Witten invariant. When
\begin{equation}\label{equa: zero dimensional FSW moduli space}
d(\mathcal L,B)
=\frac{c_1(L_b)^2-2\chi(X)-3\sigma(X)}{4}
+\dim_{\mathbb R}B=0,
\end{equation}
a generic parameter family $(g,\eta)$ in the
chosen chamber defines a compact,
zero-dimensional Seiberg--Witten moduli space.
We define the mod-$2$ family invariant by
\[
\FSW_{\mathcal L}(\mathcal X)
:=\#\mathcal M^{\mathcal L}_{g,\eta}\in\mathbb Z_2.
\]

Now we focus on the case that $B=S^n$ with $n\ge2$ and $b^+(X)>1$. Since $S^n$ is simply connected, there are identifications
$H^2(X_b;\mathbb R)\cong H^2(X;\mathbb R)$.
The fiberwise positive cones of $X_b$ form a trivial bundle
$S^n\times\mathbf K(X)$, where 
\begin{equation}
\mathbf K(X):=\{a\in H^2(X;\mathbb R)\mid a^2>0\}.
\end{equation}
By \cite{LiLiu2001}\cite{DuLi2026I}, for each $(g,\eta)\in\Gamma^*_{\mathcal{L}}$, there is a map
\[
S^n\rightarrow \mathbf K(X),\qquad Pr_{\mathcal L}(b)
:=H^+_{g_b}(\eta_b)+2\pi H^+_{g_b}(c_1(L_b))
\]
The chamber of $(g,\eta)$ is determined by the homotopy class of this map.
Following [Lemma~2.3 and Definition~2.5, \cite{Lin2022}],
the canonical chamber is the chamber corresponding to the homotopy class
of constant maps.

Set $\mathbf K_{a_0}(X)=\{a\in \mathbf K(X)|a_0\cdot a>0\}$. It's easy to see that $\mathbf K_{a_0}(X)$ is contractible for each $a_0\in \mathbf K(X)$. Then we have the following Lemma. 
\begin{lem}\label{lem:canonical-chamber}\cite{MunozEchanez25}
If
$Pr_{\mathcal L}(b)\in \mathbf{K}_{a_0}$ for every $b\in S^n$, then $(g,\eta)$ lies in the canonical chamber.
\end{lem}


We recall the family Seiberg--Witten invariant of a spherical family. Let $r\ge 1$ and let
\[
f:S^{2r-1}\longrightarrow\Diff_0(X), \quad t\rightarrow f_t
\]
be a smooth based map. Its clutching family is
\begin{equation}\label{eq:clutching-family}
\pi_f:\mathcal X_f
:=(D^{2r}_+\times X)
\cup_{f}(D^{2r}_-\times X)
\longrightarrow S^{2r},
\end{equation}
where $(t,x)_+\sim(t,f_t(x))_-$
for $t\in S^{2r-1}$.

Fix a $\spinc$ structure $\mathfrak s$ on $X$
such that
\[
d(\mathfrak s)
=\frac{c_1(\det\mathfrak s)^2
-2\chi(X)-3\sigma(X)}{4}
=-2r,
\]
and suppose that $\mathfrak s$ extends to a family $\spinc$
structure $\mathcal L_f$ on the vertical tangent bundle
$T_{\mathcal X_f/S^{2r}}$. \begin{defi}\label{defi:FSW-loop-diffeomorphisms}
For such a map $f$, define
\begin{equation}\label{eq:FSW-loop-diffeomorphisms}
\FSW_{\mathfrak s}([f])
:=\FSW_{\mathcal L_f}(\mathcal X_f)
=\#\mathcal M^{\mathcal L_f}_{g,\eta}\pmod2
\in\mathbb Z_2,
\end{equation}
where $(g,\eta)$ is a generic parameter family in the
canonical chamber.
\end{defi}
Any two extensions of $\mathfrak s$ differ by twisting with a line bundle pulled back from $S^{2r}$, which does not change the value of $\FSW_\mathfrak{s}([f])$ by \cite[Lemma~2.7]{Lin2022}. It is an invariant of the homotopy class of $f$ in $\Diff(X)$;
see \cite{Lin2022}\cite{AucklyRuberman2025}.

\subsubsection{Relative family Seiberg--Witten invariants on disk families}\label{subsubsec: relative FSW on disk}
Let $\mathfrak{s}$ be a $\spinc$ structure of $TX$. Over a smoothly trivial family $X\times \mathbb D^{n}$, restriction defines a one to one correspondence between $\spinc$ structures of $X\times \mathbb{D}^n$ and the $\spinc$ structure of $X$. And we can consider the parameter $(g,\eta)$ as a map from $\mathbb{D}^{n}$ to
\begin{equation}
\Gamma:=\{(g,\eta)|g\in\Met (X),\eta\text{ is a g-self-dual }2-\text{form}\},
\end{equation}
the space of parameters of Seiberg-Witten equations on $X$, and set 
\begin{equation}
\Gamma^*_{\mathfrak s}
:=\left\{(g,\eta)\in\Gamma|
H_g^+(\eta)
+2\pi H_g^+(c_1(\det\mathfrak s))
\ne0\right\}.
\end{equation}
Set
\begin{equation}\label{equa: Pr map}
Pr((g,\eta))
:=H_g^+(\eta)
+2\pi H_g^+(c_1(\det\mathfrak s))
\in\mathbf K(X). 
\end{equation}

By \cite{LiLiu2001}\cite{MunozEchanez25}, $Pr$ is a homotopy
equivalence from $\Gamma^*_{\mathfrak s}$ to $\mathbf K(X)$.
For $a_0\in\mathbf K(X)$, set
\[
\Gamma^*_{\mathfrak s}(a_0)
:=\{(g,\eta)\in\Gamma:
\langle Pr(g,\eta),a_0\rangle>0\}.
\]
\begin{lem}\label{lem: admissible extension}
Every map $p:S^{n-1}\to\Gamma^*_{\mathfrak s}(a_0)$
extends to $D^n\to\Gamma^*_{\mathfrak s}(a_0)$.
Any two such extensions are homotopic relative to the boundary
through maps into $\Gamma^*_{\mathfrak s}(a_0)$.
\end{lem}
\begin{proof}
By \cite[Lemma~3.1]{MunozEchanez25}, the restriction
$Pr:\Gamma^*_{\mathfrak s}(a_0)\to\mathbf K_{a_0}(X)$
is a fibration with contractible fibers.
The claim follows since $\mathbf K_{a_0}(X)$ is contractible. 
\end{proof}

Let $\Delta_{\mathfrak{s}}:=\{(g,\eta)\in\Gamma|\mathcal{M}^\mathfrak{s}_{g,\eta}\ne \emptyset\}$. Let $p:S^n\longrightarrow
\Gamma_{\mathfrak s}^{*}(a_0)$
be a parameter family such that $\Image(Pr\circ p)\subset \mathbf{K}_{a_0}(X)$ for some $a_0\in\mathbf{K}(X)$. Suppose that $\Image(p)\bigcap \Delta_{\mathfrak{s}}=\emptyset$ and
\[
\operatorname{d}({\mathfrak s})+n+1=0.
\]
For a generic extension $\widehat p:D^{n+1}\longrightarrow\Gamma_{\mathfrak s}^{*}(a_0)$ of $p$, as in Lemma \ref{lem: admissible extension}, $\bigcup_{b\in D^{n+1}}\mathcal{M}^{\mathfrak{s}_b}_{g_b,\eta_b}$ is a compact $0$-dimensional manifold, and we define 
\begin{equation}
Q_{p,a_0,\mathfrak{s}}
:=\#\left(\coprod_{b\in D^{n+1}}
\mathcal{M}^{\mathfrak{s}_b}_{g_b,\eta_b}\right)
\pmod 2\in\mathbb Z_2.
\end{equation}
Different admissible extensions are homotopic relative to their boundaries, so there is a cobordism between the corresponding moduli spaces.
\begin{lem}\label{lem: invariant of a loop p}
The evaluation of $Q_{p,a_0,\mathfrak{s}}$ is independent of the choice of extension $\widehat p$.
\end{lem}

\subsection{Kronheimer's invariant and Smirnov's $q-$invariant in higher dimension}
\label{subsec:3.3}
Let $(X,\omega)$ be a closed connected oriented symplectic $4-$manifold. We take an almost K\"ahler $3-$tuple $(\omega,g,J)$ on $X$, where $J$ is compatible to $\omega$ and $g$ is determined by $\omega$ and $J$. The symplectic structure $\omega$ determines a $\spinc$ structure $\mathfrak{s}_\omega$. Let $L_\epsilon$ be a line bundle of $X$ s.t. $c_1(L_\epsilon)=\epsilon$ and let $\mathfrak{s}_\epsilon$ be $\mathfrak{s}_\omega\otimes L_\epsilon$, its corresponding spinor bundles are $S^+=L_\epsilon\otimes(\Lambda^{0,0}\oplus\Lambda^{0,2})$ and $S^-=L_\epsilon\otimes\Lambda^{0,1}$. There is a preferred spin connection $A_{0}$ on $K^*_X$ determined by $g=\omega(\cdot,J\cdot)$. Following Taubes, we define 
\begin{equation}\label{equa: Taubes perturbation}
\eta_\rho=\eta_\rho(\omega,J) := -i F^+_{A_{0}} - \rho \omega.
\end{equation}
Recall that $\mathcal{S}_\omega$ is the space of symplectic forms isotopic to $\omega$. Let $\mathcal A_\omega$ be the space of almost complex structures tamed by some symplectic form in $\mathcal S_\omega$.

\begin{defi}\label{defi: kappa-effective}
We say that a nonzero class $\epsilon\in H^2(X;\mathbb{Z})$ is $\mathcal A_\omega$-noneffective 
if for any $J\in \mathcal A_\omega$, $\PD(\epsilon)$ is not represented by a non-constant $J$-holomorphic subvariety. 
\end{defi}
Write $K=c_1(K_X)$ for the canonical class. Let $\epsilon\in H^2(X;\mathbb{Z})$ be a class s.t. $\epsilon$ is $\mathcal{A}_\omega$-non effective and
\begin{equation}\label{equa: formal dimension condition}
\epsilon^2-K\cdot
\epsilon=-2k.
\end{equation}
Let $\{\omega_t|t\in S^{2k-1}\}\subseteq \mathcal S_{\omega}$ be a family of symplectic forms based at $\omega$, and let $\alpha\in\pi_{2k-1}(\mathcal{S}_\omega,\omega)$ be the homotopy class represented by it. Then there is a family of almost K\"ahler metrics $\{g_t=\omega_t(\cdot,J_t\cdot)|t\in S^{2k-1}\}$. Following the construction in \eqref{equa: Taubes perturbation}, there is a smooth family $$\mathfrak T_\rho(\omega_t,J_t):=(g_t=\omega_t(\cdot,J_t\cdot),\eta_\rho(t)),\quad t\in S^{2k-1}$$ 
of parameters in $\Gamma$. Let $\gamma_\rho$ be the image of $\mathfrak T_\rho$, when $\rho$ is large enough, $\gamma_\rho$ is contained in $\Gamma^*_{\mathfrak{s}_\epsilon}$ since $\omega_t$ is $g_t$-self-dual for each $t\in S^{2k-1}$. And by [Theorem $1.3$, \cite{Taubes1996}], $\gamma_\rho$ is disjoint from $\Delta_{\mathfrak{s}_\epsilon}$. Also, when $\rho$ is large enough, $Pr(\gamma_\rho)$ is contained in $\mathbf{K}_{-[\omega]}(X)$ since 
$$\langle[\omega_t],[\omega]\rangle=[\omega]^2>0,\quad \forall t\in S^{2k-1}.$$
Thus, there is a family $D=\{(g_b,\eta_b)|b\in\mathbb{D}^{2k}\}\subset \Gamma^*_{\mathfrak s_\epsilon}$ extending $\gamma_\rho$ and $Pr(D)\subset \mathbf{K}_{-[\omega]}$.
Following \cite{Smirnov2023}, 
we call $D$ an admissible extension. When $D$ is perturbed to a generic parameter family,
the moduli space $\bigcup_{b\in\mathbb{D}^{2k}}\mathcal{M}^{\mathfrak{s}_\epsilon}_{(g_b,\eta_b)}$ is a smooth $0$-dimensional smooth manifold.
\begin{defi}\label{defi:Kronheimer invariant}
Kronheimer's invariant is defined as 
\[
Q_\epsilon(\alpha)
:=\#\left(\bigcup_{b\in\mathbb D^{2k}}
\mathcal M^{\mathfrak s_\epsilon}_{g_b,\eta_b}\right)\pmod 2.
\] 
\end{defi}
\begin{lem}\label{lem: Q is a homomorphism}
$Q_\epsilon$ is a homomorphism from $\pi_{2k-1}(\mathcal{S}_\omega,\omega)$ to $\mathbb{Z}_2$.
\end{lem}
\begin{proof}
By \cite{Kronheimer1998}, the family Seiberg-Witten evaluation of $\bigcup_{b\in\mathbb{D}^{2k}}\mathcal{M}^{\mathfrak{s}_\epsilon}_{(g_b,\eta_b)}$ does not depend on the choice of representatives of $\alpha$, large numbers $\rho$ and smooth family $\{J_t|t\in S^{2k-1}\}$. And it's independent of admissible extensions by Lemma \ref{lem: invariant of a loop p}.  

To verify additivity, let $\alpha_1,\alpha_2\in
\pi_{2k-1}(\mathcal{S}_\omega,\omega)$. We choose $\{\omega_t|t\in S^{2k-1}\}$, $\{\omega_t'|t\in S^{2k-1}\}$ in $(\mathcal{S}_\omega,\omega)$, which represent $\alpha_1$ and
$\alpha_2$ respectively. We take almost complex families $\{J_t|t\in S^{2k-1}\}$, $\{J_t'|t\in S^{2k-1}\}$ which are fiberwise compatible
to $\{\omega_t|t\in S^{2k-1}\}$, $\{\omega_t'|t\in S^{2k-1}\}$ in $(\mathcal{S}_\omega,\omega)$.
The representative of $\alpha_1+\alpha_2$ is represented by 
\[
S^{2k-1}\longrightarrow S^{2k-1}\vee S^{2k-1}\rightarrow \{\omega_t|t\in S^{2k-1}\}\vee\{\omega_t'|t\in S^{2k-1}\}
\]
and $Q_\epsilon(\alpha_1+\alpha_2)$ is equal to the Seiberg-Witten invariant on an admissible extension of $$\{\mathfrak{T}_\rho(\omega_t,J_t)|t\in S^{2k-1}\}\bigcup\{\mathfrak{T}_\rho(\omega_t',J_t')|t\in S^{2k-1}\}.$$
Let $D,D'\subset \Gamma^*_{\mathfrak s_\epsilon}$ be admissible extension of $\{\mathfrak{T}_\rho(\omega_t,J_t)|t\in S^{2k-1}\}$ and $\{\mathfrak{T}_\rho(\omega_t',J_t')|t\in S^{2k-1}\}$, then their union is the desired admissible extension. Hence
\[
Q_\epsilon(\alpha_1+\alpha_2)
=Q_\epsilon(\alpha_1)+Q_\epsilon(\alpha_2)\pmod2.
\]
\end{proof}

\subsubsection{Index sets of Kronheimer invariants on $X\#n\overline{\mathbb{CP}^2}$}
Now let $X$ be $T^4$ or a $K3$ surface. With a specified non-resonant positive class $\kappa$, we define the following index sets
\begin{align}
\Delta_{T^4}&=\{\delta\in\Lambda|\delta\text{ primitive},\
\delta^2=0,\ \kappa\cdot\delta>0\},
\label{equa: definition of Delta for torus}\\
\Delta_{K3}&=\{\delta\in\Lambda|\delta^2=-2,\ \kappa\cdot\delta>0\},
\label{equa: definition of Delta for K3 surface}\\
\Delta_{X,k}&=\{\delta\in\Delta_X|(k-1)\lambda<\kappa\cdot\delta<k\lambda\}.
\label{equa: definition of Delta_k}
\end{align}
We write
$\Delta$ and $\Delta_k$ when the surface is understood.
We say $\kappa$ is non-resonant if $\kappa\cdot a\ne0$ for every $0\ne a\in\Lambda$.
\begin{lem}\label{lem:2-infinite-windows}
$\Delta_{X,k}$ is infinite for every $1\le k\le n$.
\end{lem}
\begin{proof}
The proof is essentially the same as [Lemma $10$, \cite{Smirnov2023}]. For $\Lambda=3H$, choose generators $f_1,f_2,f_3$ of a primitive
totally isotropic rank-three sublattice. All classes
$\delta_{p,q}:=f_1+pf_2+qf_3$, $p,q\in\mathbb Z$, are primitive and $\delta_{p,q}^2=0$. Non-resonance implies that $\kappa\cdot f_2$ and
$\kappa\cdot f_3$ are linearly independent over $\mathbb Q$. Thus $\{\kappa\cdot (f_1+pf_2+qf_3)\}$
is a dense subset of $\mathbb R$, and $\Delta_{X,k}$ is infinite. This proves the case when $X=T^4$.

Now let $X$ be a K3 surface and let $\Lambda=2(-E_8)\oplus 3H$.
Let $r$ be a root of a copy of $(-E_8)$ in $\Lambda$, i.e.
$r^2=-2$, and choose primitive 
$f_1,f_2$ from two distinct copies of $H$ such that $f_1^2=f_2^2=0$.

For every $p,q\in\mathbb Z$, the class $\delta_{p,q}=r+pf_1+qf_2$ satisfies that
$$\delta_{p,q}^2=r^2+p^2f_1^2+q^2f_2^2=-2.$$
By the non-resonance assumption of $\kappa$, $\kappa\cdot f_1$ and
$\kappa\cdot f_2$ are linearly independent vectors over
$\mathbb Q$, and $\{\langle \kappa,\delta_{p,q}\rangle|p,q\in\mathbb Z\}$ is dense in $\mathbb R$. The infiniteness of $\Delta_{X,k}$ follows. 
\end{proof}
We set
\begin{equation}\label{equa: definition of I_k}
I_k=\{1,...,k\}.
\end{equation}
Now let $(\tilde X,\tilde\omega)$ be $n-$point blow-up of $(X,\omega)$. We assume that the area of each exceptional divisor is equal to $\lambda$ and $n\lambda$ is small enough. Since $K_{\tilde\omega}=e_1+...+e_n$, and $$K_{\tilde\omega}-(\delta-\sum_{i\in I_k}e_i)=2\sum_{i\in I_k}e_i+\sum_{i\in\{1,...,n\}\backslash I_k}e_i-\delta,$$
the conjugate of $\mathfrak s_{\delta-\sum_{i\in I_k}e_i}$ is isomorphic to $\mathfrak{s}_{2\sum_{i\in I_k}e_i+\sum_{i\in\{1,...,n\}\backslash I_k}e_i-\delta}$.
\begin{lem}\label{lem: non-effective class for X}
For each $\delta\in \Delta_{X,k}$, both $\delta-\sum_{i\in I_k}e_i$ and $2\sum_{i\in I_k}e_i+\sum_{i\in\{1,...,n\}\backslash I_k}e_i-\delta$ are $\mathcal{A}_{\tilde\omega}$-non-effective.
\end{lem}
\begin{proof}
By the definition of $\Delta_{X,k}$, we have $\langle\tilde\omega,\delta-\sum_{i\in I_k}e_i\rangle<0$, thus $\delta-\sum_{i\in I_k}e_i$ is $\mathcal{A}_{\tilde\omega}$-non-effective. Since all the area of exceptional divisors $e_i$ are equal to each other, they are all minimal, and $E_i$ is represented by an embedded $J$-holomorphic curve for each $J\in\mathcal{A}_{\tilde\omega}$ [Lemma 1.2, \cite{Pinsonnault2006}].For each $t =1,...,n$, we have 
$$e_t\cdot (2\sum_{i\in I_k}e_i+\sum_{i\in\{1,...,n\}\backslash I_k}e_i-\delta)<0,$$
If there is $J\in\mathcal{A}_{\tilde\omega}$ and a $J$-holomorphic sub-variety in $\PD(2\sum_{i\in I_k}e_i+\sum_{i\in\{1,...,n\}\backslash I_k}e_i-\delta)$, it consists of a curve in $\PD(\sum_{i\in I_k}e_i-\delta)$ and a component in $E_i$ for each $i=1,...,n$. Applying this strategy to the curve in $\PD(\sum_{i\in I_k}e_i-\delta)$ one more time, we get a curve in $-\PD(\delta)$, which is impossible.
\end{proof}
For each $\delta\in \Delta_{T^4,k}$, 
\begin{equation}\label{equa: formal dimension of FSW for elliptic curves}
d(\mathfrak s_{\delta-\sum_{i\in I_k}e_i})=d(\mathfrak s_{K_{\tilde\omega}-(\delta-\sum_{i\in I_k}e_i)})=(\delta-\sum_{i\in I_k}e_i)^2-\langle \delta-\sum_{i\in I_k}e_i,\sum_{i=1}^ne_i \rangle= -2k, 
\end{equation} 
so the Kronheimer's invariants $Q_{\delta-\sum_{i\in I_k}e_i}$ and $Q_{K_{\tilde\omega}-(\delta-\sum_{i\in I_k}e_i)}$ are well-defined homomorphisms from $\pi_{2k-1}(\mathcal{S}_{\tilde \omega},\tilde\omega)$ to $\mathbb Z_2$. We define
$$q_{\delta-\sum_{i\in I_k}e_i}=Q_{\delta-\sum_{i\in I_k}e_i}-Q_{K_{\tilde\omega}-(\delta-\sum_{i\in I_k}e_i)}:\pi_{2k-1}(\mathcal S_{\tilde\omega},\tilde\omega)\rightarrow\mathbb{Z}_2.$$
For each $\delta\in \Delta_{K3,k}$, 
\begin{equation}\label{equa: formal dimension of FSW for rational curves}
d(\mathfrak s_{\delta-\sum_{i\in I_k}e_i})=d(\mathfrak s_{K_{\tilde\omega}-(\delta-\sum_{i\in I_k}e_i)})=(\delta-\sum_{i\in I_k}e_i)^2-\langle \delta-\sum_{i\in I_k}e_i,\sum_{i=1}^ne_i \rangle= -2k-2, 
\end{equation} 
so the Kronheimer's invariants $Q_{\delta-\sum_{i\in I_k}e_i}$ and $Q_{K_{\tilde\omega}-(\delta-\sum_{i\in I_k}e_i)}$ are well-defined homomorphisms from $\pi_{2k+1}(\mathcal{S}_{\tilde \omega},\tilde\omega)$ to $\mathbb Z_2$. We define
$$q_{\delta-\sum_{i\in I_k}e_i}=Q_{\delta-\sum_{i\in I_k}e_i}-Q_{K_{\tilde\omega}-(\delta-\sum_{i\in I_k}e_i)}:\pi_{2k+1}(\mathcal S_{\tilde\omega},\tilde\omega)\rightarrow\mathbb{Z}_2.$$
Now, we'll introduce some properties of $q-$invariant.
\begin{lem}\label{lem:spinc-extension-clutching}
Let $\tilde X$ be either $K3\#n\overline{\mathbb{CP}}^{\,2}$ or $T^4\#n\overline{\mathbb{CP}}^{\,2}$, with $n\geq1$.
For $r\geq 2$, let
$f:S^{2r-1}\to\operatorname{Diff}_0(\tilde X)$ be a smooth based map and set
\[
\pi_f:E_f=(D^{2r}_+\times \tilde X)\bigsqcup(D^{2r}_-\times \tilde X)/\sim\longrightarrow S^{2r},
\qquad
(t,x)_+\sim(t,f_t(x))_-.
\]
Every $\operatorname{spin}^c$ structure $\mathfrak s$ on $\tilde X$
extends to a $\operatorname{spin}^c$ structure
$\widetilde{\mathfrak s}$ on $T_{E_f/S^{2r}}$, and the extension is unique up to isomorphism.
\end{lem}
\begin{proof}
Equip $D^{2r}_{\pm}\times \tilde X$ with the product $\operatorname{spin}^c$
structure $\mathcal{L}_{\pm}$ induced by $\mathfrak s$.
After identifying the vertical tangent bundles on the overlap
by $\{df_t|t\in S^{2r-1}\}$, $\mathcal{L}_{\pm}|_{S^{2r-1}\times \tilde{X}}$ differ by a complex line bundle
$\mathcal L\to S^{2r-1}\times \tilde X$ such that 
$\mathcal L|_{\{t\}\times \tilde X}$ is trivial for each $t\in S^{2r-1}$. As $r\geq2$, the restriction induces an isomorphism $H^2(S^{2r-1}\times \tilde X;\mathbb Z)
\xrightarrow{\;\cong\;} H^2(\tilde X;\mathbb Z)$.
Since $\mathcal L|_{\{t\}\times \tilde X}$ is trivial, $\mathcal L$ is trivial, and the two family $\spinc$-structures glue to a $\spinc$ structure on $T_{E_f/S^{2r}}$.

Let $j_\pm$ be inclusions $D^{2r}_\pm\times\widetilde X\hookrightarrow E_f$, and let $k_{\pm}$ be inclusions $S^{2r-1}\times \tilde{X}\hookrightarrow D^{2r}_\pm\times\widetilde X$. Under the identification by K\"unneth formula, we have a surjective homomorphism
$$H^1(D^{2r}_+\times\widetilde X;\mathbb Z)
\oplus H^1(D^{2r}_-\times\widetilde X;\mathbb Z)
\xrightarrow{\;k_+^*-k_-^*\;}
H^1(S^{2r-1}\times\widetilde X;\mathbb Z),\quad (a,b)\rightarrow a-b.$$ Applying the Mayer--Vietoris sequence to $(E_f,D^{2r}_+\times\tilde X, D^{2r}_{-}\times \tilde X, S^{2r-1}\times \tilde{X})$, we have
\begin{equation}\label{equa: MV sequence and spinc structure}
\begin{aligned}
0\longrightarrow H^2(E_f;\mathbb Z)
&\xrightarrow{\;(j_+^*,j_-^*)\;}
H^2(D^{2r}_+\times\widetilde X;\mathbb Z)
\oplus H^2(D^{2r}_-\times\widetilde X;\mathbb Z)
\\
&\xrightarrow{\;k_+^*-k_-^*\;}
H^2(S^{2r-1}\times\widetilde X;\mathbb Z)
\longrightarrow0.
\end{aligned}
\end{equation}
Since $r\ge2$ and each $f_t$ is isotopic to the identity,
$k_+^*-k_-^*$ can be identified with $$H^2(\tilde{X};\mathbb Z)\oplus H^2(\tilde{X};\mathbb Z)\rightarrow H^2(\tilde{X};\mathbb Z),\quad (a,b)\rightarrow a-b.$$
By exactness of \eqref{equa: MV sequence and spinc structure}, $(j_+^*,j_-^*)$ is an isomorphism to the diagonal group
\[
\{(a,a):a\in H^2(\widetilde X;\mathbb Z)\}.
\]
The projection to the first factor is an isomorphism on the diagonal, so $j_+^*$ is an isomorphism. Let $i_+:\widetilde X\hookrightarrow
D_+^{2r}\times\widetilde X$ be the inclusion of a fiber, $i_+^*$ is an isomorphism between the cohomology groups.
Set $\iota=j_+\circ i_+$,
\[
\iota^*=i_+^*\circ j_+^*:
H^2(E_f;\mathbb Z)\xrightarrow{\;\cong\;}
H^2(\widetilde X;\mathbb Z).
\]

Let $\widetilde{\mathfrak s}$ and
$\widetilde{\mathfrak s}'$ be two extensions of $\mathfrak s$, then $\widetilde{\mathfrak s}'
\cong\widetilde{\mathfrak s}\otimes\mathcal L$ for a line bundle $\mathcal L\to E_f$ such that 
$\iota^*c_1(\mathcal L)=0$. The injectivity of $\iota^*$ implies $c_1(\mathcal L)=0$ and $\mathcal L$ is trivial.
\end{proof}
Let $r$ be $k$ when $X=T^4$ and $r=k+1$ when $X$ is a $K3$ surface.
\begin{lem}\label{lem: reduction of q to monodromies}
For $\delta\in \Delta_{X,k}$ and $r\ge 2$, set $\epsilon=\delta-\sum_{i\in I_k}e_i$, the composition
\[
q_{\epsilon}\circ\p_*:
\pi_{2r-1}(\Diff_0(\widetilde X),\mathrm{id})
\longrightarrow
\pi_{2r-1}(\mathcal S_{\widetilde\omega},\widetilde\omega)\rightarrow \mathbb Z_2
\]
is trivial.
\end{lem}
\begin{proof}
We follow the argument of
[Lemma~6, \cite{Smirnov2023}]. Let \{$f_t|t\in S^{2r-1}$\} be a family of diffeomorphisms based at $id$ and set $\tilde\omega_t=f_t^*\tilde\omega$. The family $\{\tilde\omega_t\mid t\in S^{2r-1}\}$ represents $\alpha$. Let
\[
E_f=(D^{2r}_+\times\widetilde X)
\bigcup(D^{2r}_-\times\widetilde X)/\sim
\longrightarrow S^{2r},
\qquad
(t,x)_+\sim(t,f_t(x))_-.
\]
Set $K=K_{\tilde\omega}$ and $\epsilon=\delta-\sum_{i\in I_k}e_i$. Then $K-\epsilon$ is $2\sum_{i\in I_k}e_i+\sum_{i\in\{1,...,n\}\backslash I_k}e_i-\delta$. By Lemma~\ref{lem:spinc-extension-clutching}, both
$\mathfrak s_{\epsilon}$ and $\mathfrak s_{K-\epsilon}$ extends to
$\spinc$ structures
$\mathcal L_{\epsilon}$, $\mathcal{L}_{K-\epsilon}$ on the vertical tangent bundle of $E_f$. By the uniqueness of extension, $\overline{\mathcal{L}}_\epsilon\cong \mathcal{L}_{K-\epsilon}$.

Let $J$ be an almost complex structure compatible with $\widetilde\omega$, let $g=\tilde\omega(\cdot,J\cdot)$, and set
$J_t=f_t^*J$ and $g_t=f_t^*g$.
The corresponding perturbation $2$-forms $\eta_\rho(\tilde\omega_t,J_t)$ satisfy $$\eta_\rho(\tilde\omega_t,J_t)=f_t^*\eta_\rho(\tilde{\omega},J),\quad t\in S^{2r-1}.$$ 
We take an admissible extension $\{(g_b,\eta_b)|b\in D^{2r}_+\}$ of $\{(g_t,\eta_\rho(\tilde\omega_t,J_t))\}$, and take a trivial perturbation $\{(g,\eta_\rho(\tilde\omega,J))\}$ on $D^{2r}_-$, they form a family of perturbation $\{(g_z,\eta_z)|z\in S^{2r}\}$ on $E_f$. 

Since $S^{2r}$ is simply connected for each $r\ge 2$, the chamber of $\{(g_z,\eta_z)|z\in S^{2r}\}$ is determined by the homotopy class of 
\[
Pr_{\mathcal{L}_\epsilon}:S^{2r}\rightarrow \mathbf K(\widetilde X),\qquad b
\longrightarrow H^+_{g_b}(\eta_b)+2\pi H^+_{g_b}(c_1(L_b)),
\]
which is contained in $\mathbf{K}_{-[\tilde\omega]}$ by the admissibility of $\{(g_b,\eta_b)|b\in D^{2r}_+\}$. Thus, $Q_\epsilon(\alpha)
$ is equal to $\FSW_{\mathcal L_\epsilon}(E_f)$. Since $\mathcal L_{K-\epsilon}\cong\overline{\mathcal L_\epsilon}$ $c_1(det(\mathcal{L}_{K-\epsilon}))=-c_1(det(\mathcal{L}_{\epsilon}))$, we have \[ Pr_{\mathcal L_{K-\epsilon}}(g_z,-\eta_z) =-Pr_{\mathcal L_\epsilon}(g_z,\eta_z) \in\mathbf K_{[\tilde\omega]}, \] 
which also determines the canonical chamber. Applying the same construction to
$\mathfrak s_{K-\epsilon}$, we can identify
$Q_{K-\epsilon}(\alpha)$ with $\FSW_{\mathcal L_{K-\epsilon}}(E_f)$. By the charge conjugation symmetry,
\begin{equation}\label{eq:charge-conjugation-reduction}
\begin{aligned}
Q_\epsilon(\alpha)
&= \FSW_{\mathcal L_\epsilon}(E_f)\\
&=\FSW_{\mathcal L_\epsilon}
(E_f;\{(g_z,\eta_z)\})\\
&=\FSW_{\mathcal L_{K-\epsilon}}
(E_f;\{(g_z,-\eta_z)\})\\
&=\FSW_{\mathcal L_{K-\epsilon}}(E_f)\\
&=Q_{K-\epsilon}(\alpha).
\end{aligned}
\end{equation}
Consequently, $q_\epsilon(\alpha)=0$.
\end{proof}
\begin{defi}
For $r\ge2$ and $\delta\in\Delta_{X,k}$, we define 
\begin{align}
\bar{q}_{\delta-\sum_{i\in I_k}e_i}:\Ker(\pi_{2r-2}(\Symp_{st}(\widetilde X,\widetilde\omega))\rightarrow \pi_{2r-2}(\Diff_0(\widetilde X)))&\longrightarrow \mathbb Z_2,\nonumber\\
\beta &\longrightarrow q_{\delta-\sum_{i\in I_k}e_i}(\alpha),
\end{align}
where $\alpha\in\pi_{2r-1}(\mathcal S_{\widetilde\omega},\widetilde\omega)$ satisfies $\partial_*\alpha=\beta$.
\end{defi}

\subsection{Family blowups and Seiberg--Witten equations on K\"ahler surfaces}
\label{subsec: FSW on Kahler family and blowup}

Let $\pi:\mathcal X\to(B,0)$ be a proper holomorphic submersion with
compact connected surface fibers, where $B\subset\mathbb C^n$ is a
polydisk. Write $X_b=\pi^{-1}(b)$ and $X=X_0$. In this section, we relate the family Seiberg-Witten equation on $\mathcal{X}$, to the family Seiberg-Witten equation of a $D^2\times B$ family of $\tilde X$.

Choose a smooth
trivialization over $B$, we can identify $\mathcal{X}$ with a smooth family $\{J_b|b\in B\}$ of complex structures, and let $\{(\omega_b,J_b)|b\in B\}$ be a
smooth family of K\"ahler structures. Fix a smooth complex
line bundle $\mathcal L_\alpha$ on $\mathcal X$ with fiberwise first
Chern class $\alpha$. As before, $\mathfrak s_\alpha$ denotes the $\spinc$ structure determined by 
$\mathfrak s_{\omega_b}\otimes\mathcal L_\alpha|_{X_b}$. Thus
\[
c_1(\det\mathfrak s_\alpha)=2\alpha-K_{X_b},
\qquad
d_X(\mathfrak s_\alpha)=\alpha^2-K_X\cdot\alpha.
\]
We use the K\"ahler metrics $g_b=\omega_b(\cdot,J_b\cdot)$ and Taubes' perturbations
$\eta_\rho(\omega_b,J_b)$ introduced above. When $\rho$ is large enough for the family, $\mathfrak T_\rho(\omega_b,J_b)=(g_b,\eta_\rho(\omega_b,J_b))$ will be contained in $\Gamma^*_{\mathfrak{s}_\alpha}$.

For a smooth curve $C\subset X$ such that $[C]=\PD(\alpha)$, let $\Theta_X$ be the sheaf of
holomorphic vector fields and write
\[
N=\mathcal N_{C|X}=T_X|_C/T_C,
\qquad
r_*:H^1(X,\Theta_X)\longrightarrow H^1(C,N)
\]
for normal restriction. Let
$\KS:T_0B\to H^1(X,\Theta_X)$ be the Kodaira--Spencer map. By \cite{Taubes1995}, a smooth curve $C\subset X_0$ represents a solution of $\bigcup_{b\in B}\mathcal{M}^{\mathfrak{s}_\alpha}_{\eta_\rho(\omega_b,J_b)}$ when $\rho$ is big enough. And by \cite{FriedmanMorgan1995}\cite{Kronheimer1998}\cite{Smirnov2020}, the solution corresponding to $C$ is regular if
\begin{equation}\label{eq:34-regularity-map}
r_*\circ\KS:T_0B\longrightarrow H^1(C,N)
\end{equation}
is surjective. 

\begin{assum}\label{assum: assumption of regularity}
Assume that $C$ is the unique effective divisor in class $\alpha$
among all fibers of $\mathcal X$. Assume that
\begin{equation}\label{eq:34-normal-rigidity}
H^0(C,N)=0,
\qquad
d_X(\mathfrak s_\alpha)+2n=0,
\end{equation}
and that \eqref{eq:34-regularity-map} is surjective. 
\end{assum}

Let $\Delta\subset\mathbb C$ be a disk centred at $0$, and let
$F:B\times\Delta\to\mathcal X$ be a jointly holomorphic map such that
$\pi\circ F=\operatorname{pr}_B$ and $F(0,0)=p\in C$. Assume that
\begin{equation}\label{eq:34-transverse-centre}
\nu:T_0\Delta\longrightarrow N|_p,
\qquad
\nu(w)=d(F|_{\{0\}\times\Delta})_0(w)\bmod T_pC
\end{equation}
is nonzero, and hence an isomorphism of complex lines.

Set a holomorphic family $\mathcal{X}_\Delta$ and a section $\sigma$ to it
\[
\mathcal X_\Delta=(B\times\Delta)\times_B\mathcal X,
\qquad
\pi_\Delta:\mathcal X_\Delta\longrightarrow B\times\Delta,
\qquad
\sigma(b,z)=((b,z),F(b,z)).
\]
Blowing up along the image of $\sigma$ gives
\[
\Pi:\widetilde{\mathcal X}
=\operatorname{Bl}_{\sigma(B\times\Delta)}\mathcal X_\Delta
\longrightarrow\mathcal X_\Delta,\quad \widetilde\Pi=\pi_\Delta\circ\Pi:
\widetilde{\mathcal X}\to B\times\Delta.
\]
A fiber of the $\tilde\Pi$
 is
$\widetilde X_{b,z}=\operatorname{Bl}_{F(b,z)}X_b$, and the restriction
$\pi_{b,z}=\Pi|_{\widetilde X_{b,z}}:
\widetilde X_{b,z}\to X_b$
is the fiberwise blowdown. Let $E$ be the relative exceptional divisor and $e$ its fiberwise cohomology class.


Set the line bundle $\widetilde{\mathcal L}_\alpha
=\Pi^*\operatorname{pr}_{\mathcal X}^*\mathcal L_\alpha
\otimes\mathcal O_{\widetilde{\mathcal X}}(-E),$ its fiberwise first Chern class is $\alpha-e$. Shrinking $B$ and $\Delta$ if necessary, there is a family $\{\widetilde\omega_{b,z}\}$ of K\"ahler forms  of class
$\pi_{b,z}^*[\omega_b]-\varepsilon e$ for a fixed sufficiently small
$\varepsilon>0$.
\begin{thm}
\label{thm:regularity of FSW solutions on blow up family}
With Assumption \ref{assum: assumption of regularity} satisfied, for all sufficiently large $\rho$, the family
moduli space
\[
\bigcup_{(b,z)\in B\times\Delta}
\mathcal M^{\mathfrak s_{\alpha-e}}_{
 \eta_\rho(\widetilde\omega_{b,z},\widetilde J_{b,z})}
\]
consists of a unique solution which is regular, represented by the strict
transform $\widetilde C\subset\widetilde X_{0,0}$. In other word, the restriction of the Kodaira--Spencer map at $(0,0)$
is surjective onto
$H^1(\widetilde C,\mathcal N_{\widetilde C|\widetilde X_{0,0}})$.
\end{thm}

Let $\pi_0:\widetilde X\to X$ be the
blowdown at $p\in C$. Identify the strict transform $\widetilde C$
with $C$ by $f:=\pi_0|_{\tilde{C}}$, and write $\mathfrak m_p$ for the ideal sheaf of $p$. and $d\pi_0$ identifies
\begin{equation}\label{eq:34-normal-transform}
\mathcal N_{\widetilde C|\widetilde X}\cong f^*(N(-p)),
\qquad N(-p)=N\otimes\mathcal O_C(-p).
\end{equation}
\begin{lem}\label{lem:34-analytic-blowup}
With the identifications above, there are exact sequences
\begin{equation}\label{equa: tangent cohomology for surfaces}
H^0(X,\Theta_X)\xrightarrow{\operatorname{ev}_p}T_pX
\xrightarrow{\delta_p}H^1(\widetilde X,\Theta_{\widetilde X})
\xrightarrow{\beta}H^1(X,\Theta_X)\longrightarrow0
\end{equation}
and
\begin{equation}\label{equa: normal cohomology of curves}
H^0(C,N)\xrightarrow{\operatorname{ev}_p}N|_p
\xrightarrow{\partial_N}H^1(C,N(-p))
\xrightarrow{\gamma}H^1(C,N)\longrightarrow0.
\end{equation}
Let $r_*: H^1(X,\Theta_X)\rightarrow H^1(C,N)$ and $\widetilde r_*: H^1(\tilde{X},\Theta_{\tilde{X}})\rightarrow H^1(\tilde C,N_{\tilde{C}|\tilde{X}})$ be restriction maps, $\operatorname{pr}_N:T_pX\to N|_p=T_pX/T_pC$ is the quotient map. Then
\begin{equation}\label{eq:34-restriction-naturality}
\gamma\circ\widetilde r_*=r_*\circ\beta,
\qquad
\widetilde r_*\circ\delta_p
=\partial_N\circ\operatorname{pr}_N,
\end{equation}
In particular,
\begin{equation}\label{eq:34-kernel-surjectivity}
\widetilde r_*(\ker\beta)=\ker\gamma.
\end{equation}
\end{lem}
\begin{proof}
For one point blowup in the complex-analytic category, the standard
tangent-sheaf identities are
\begin{equation}\label{eq:34-tangent-direct-images}
(\pi_0)_*\Theta_{\widetilde X}=\mathfrak m_p\Theta_X,
\qquad
R^1(\pi_0)_*\Theta_{\widetilde X}=0;
\end{equation}
see the Proof of [Proposition~2.6\cite{Peternell2020}], where these
identities are used for compact complex manifolds.
The first identification is induced by $d\pi_0$.
By the five-term exact sequence of the Leray spectral sequence, we have
\[
\begin{aligned}
&0\longrightarrow E^{1,0}_2=H^1(X,(\pi_0)_*\Theta_{\widetilde X})
\longrightarrow H^1(\widetilde X,\Theta_{\widetilde X})
\longrightarrow E_2^{0,1}=H^0(X,R^1(\pi_0)_*\Theta_{\widetilde X})\cong\{0\}\\
&\xrightarrow{d_2}E_2^{2,0}
\longrightarrow H^2(\widetilde X,\Theta_{\widetilde X}),
\end{aligned}
\]
thus $H^1(X,\mathfrak m_p\Theta_X)
\cong H^1(\widetilde X,\Theta_{\widetilde X})$.
Under this identification, \eqref{equa: tangent cohomology for surfaces} is induced by the evaluation sequence
$$0\to\mathfrak m_p\Theta_X\to\Theta_X\to(T_pX)_p\to 0,$$
where $(V)_p$ denotes the skyscraper sheaf with fiber $V$ at $p$. This defines $\delta_p$ and $\beta$.

Choose a coordinate system $(x,y)$ near $p$ with $C=\{x=0\}$.
Let $q$ be the unique point of
$\widetilde C\cap\pi_0^{-1}(p)$.
There are holomorphic coordinates $(u,y)$ on a neighbourhood of $q$
in which
\[
\pi_0(u,y)=(uy,y),
\qquad
\widetilde C=\{u=0\}.
\]
In these coordinates, the restriction
$f=\pi_0|_{\widetilde C},\ (0,y)\mapsto(0,y)$ is a biholomorphism from $\tilde{C}$ to $C$. $d\pi_0$ induces a morphism from $N_{\tilde{C}|\tilde{X}}$ to $N$, and its local expression is$[\partial_u]\longmapsto y[\partial_x]$. This identifies $N_{\widetilde C/\widetilde X}$ with $ N(-p)$, and the morphism induced by $d\pi_0$
is identified with natural inclusion $N(-p)\hookrightarrow N$. The long exact sequence corresponding to 
\[
0\longrightarrow N(-p)\longrightarrow N
\xrightarrow{\mathrm{ev}_p}(N|_p)_p\longrightarrow0
\]
therefore yields \eqref{equa: normal cohomology of curves}.

Write $i:C\hookrightarrow X$. Restrictions of the sheaves give the following diagram of sheaves on $X$
with exact rows:
\begin{equation}\label{eq:34-sheaf-diagram}
\begin{CD}
0 @>>> \mathfrak m_p\Theta_X @>>> \Theta_X
  @>>> (T_pX)_p @>>> 0 \\
@. @V{R}VV @V{r}VV @V{\operatorname{pr}_N}VV @. \\
0 @>>> i_*N(-p) @>>> i_*N
  @>>> (N|_p)_p @>>> 0.
\end{CD}
\end{equation}
Here, $r$ is normal restriction to $C$, and $R$ is its restriction
to $\mathfrak m_p\Theta_X$, taking values in $i_*N(-p)$. The left square commutes by the identification with $d\pi_0$ just established, and the right square commutes by the property of the evaluation at $p$. The isomorphism $H^1(X,\mathfrak m_p\Theta_X)
\cong H^1(\widetilde X,\Theta_{\widetilde X})$ 
identifies $H^1(R)$ with $\widetilde r_*$, and for sheaves supported on $C$, the same comparison is the
cohomology identification induced by the biholomorphism $f$.
Taking cohomology of \eqref{eq:34-sheaf-diagram} and using naturality
of connecting homomorphisms, we get \eqref{eq:34-restriction-naturality}.
Finally, by the exactness and the surjectivity of $\operatorname{pr}_N$, we have
\[
\widetilde r_*(\ker\beta)
=\widetilde r_*(\operatorname{im}\delta_p)
=\operatorname{im}(\partial_N\operatorname{pr}_N)
=\operatorname{im}\partial_N
=\ker\gamma.
\]
\end{proof}
Choose an open coordinate neighbourhood $U_0\subset X$ of $p$, with
coordinates $z=(z_1,z_2)$ satisfying $z(p)=0$. Let $U\subset U_0$ be an open neighbourhood of $p$ and set 
$\mathcal Y:=\{\operatorname{Bl}_qX|q\in U\}$. Let $v$ be a vector in $T_pU$, we can regard $v$ as a global section of the skyscraper sheaf $(T_pX)_p$. Then $\delta_p(v)$ is its image in the connecting homomorphism of the long exact sequence
\[
H^0(X,\mathfrak m_p\Theta_X)\longrightarrow H^0(X,\Theta_X)\rightarrow H^0(X,(T_pX)_p)\rightarrow H^1(X,\mathfrak m_p\Theta_X) \rightarrow...
\]
By the proof of [Theorem~4.2, \cite{PeneginiPignatelli2022}], the Kodaira--Spencer map $\KS_p(v)$ of $\mathcal Y$ at $p$ is its image under the connecting homomorphism
associated with
\[
0\rightarrow H^0(X,(\pi_0)_*\Theta_{\widetilde X})
\rightarrow H^0(X,\Theta_X)
\rightarrow H^0(X,(T_pX)_p)\rightarrow H^1(X,(\pi_0)_*\Theta_{\widetilde X})\rightarrow ...
\]
Under the identifications in
Lemma~\ref{lem:34-analytic-blowup}, $\KS_p$ is identified with
$\delta_p$.
Consequently, we have
\[
\operatorname{Im}\KS_p=\ker\beta,
\qquad
\ker\KS_p
=\operatorname{Im}\bigl(H^0(X,\Theta_X)
\xrightarrow{\operatorname{ev}_p}T_pX\bigr).
\]And now we are ready to prove the induction of transversality of solution of family Seiberg--Witten equations.
\begin{proof}[Proof of
Theorem~\ref{thm:regularity of FSW solutions on blow up family}]

Set $F_0(z)=F(0,z)$. By \eqref{eq:34-transverse-centre}, after
shrinking $\Delta$, we have $F_0(\Delta)\subset U$ and $F_0^{-1}(C)=\{0\}$. For $(b,z)\in B\times\Delta$, write
\[
\pi_{b,z}:\widetilde X_{b,z}
=\operatorname{Bl}_{F(b,z)}X_b\longrightarrow X_b
\]
let $\tilde{\mathcal X}$ be $\{\operatorname{Bl}_{F(b,z)}X_b|b\in B,z\in \Delta\}$
and let $\{\tilde{\omega}_{b,z}|(b,z)\in B\times \Delta\}$ be a smooth family of K\"ahler forms on $\tilde{\mathcal X}$. 
For sufficiently large $\rho$,elements in $\bigcup_{(b,z)\in B\times\Delta}
\mathcal M^{\mathfrak s_{\alpha-e}}_{
\eta_\rho(\widetilde\omega_{b,z},\widetilde J_{b,z})}$ are in one-to-one correspondence with effective divisors in $\alpha-e$. 
The blowdown of any such divisor is an effective divisor
in class $\alpha$ on some $X_b$. By the uniqueness hypothesis,
it must be $C\subset X_0$.  Since $F_0^{-1}(C)=\{0\}$ and
$C$ is smooth, its proper transform
$\widetilde C\subset\widetilde X_{0,0}$ is the unique effective divisor in the class. Hence the family moduli space consists of exactly one solution.

It remains to prove regularity of the point in $\bigcup_{(b,z)\in B\times\Delta}
\mathcal M^{\mathfrak s_{\alpha-e}}_{
\eta_\rho(\widetilde\omega_{b,z},\widetilde J_{b,z})}$. 
Write
\[
\KS:T_0B\longrightarrow H^1(X,\Theta_X),
\qquad
\KS':T_0B\oplus T_0\Delta
\longrightarrow H^1(\widetilde X,\Theta_{\widetilde X}),
\]
for the Kodaira-Spencer map of $\mathcal{X}$ and $\tilde{\mathcal{X}}$, and set
\[
L=\widetilde r_*\circ\KS':
T_0B\oplus T_0\Delta
\longrightarrow H^1(C,N(-p)).
\]
By \eqref{equa: normal cohomology of curves} and that $H^0(C,N)=0$, we have
\begin{equation}\label{equa: surjection to H^1(C,N)}
0\longrightarrow N|_p
\xrightarrow{\partial_N}H^1(C,N(-p))
\xrightarrow{\gamma}H^1(C,N)\longrightarrow0.
\end{equation}

Since $\{\pi_{b,z}(\tilde{X}_{b,z})\mid(b,z)\in B\times\Delta\}$ can be identified with the pullback family
$\operatorname{pr}_B^*\mathcal X\to B\times\Delta$,
where $\operatorname{pr}_B$ is the projection $B\times\Delta\to B$.
By naturality of the Kodaira--Spencer map, $\beta\circ\KS'(v,w)=\KS(v)$. Together with Lemma~\ref{lem:34-analytic-blowup}, this implies
\[
\gamma\circ  L(v,w)=r_*\circ\KS(v).
\]
Thus $\gamma\circ L$ is surjective, and
$L(\{0\}\oplus T_0\Delta)\subseteq\ker\gamma$.

The family $\{\tilde{X}_{0,z}|z\in\Delta\}$ is the pullback of
$\mathcal Y\to U$ by $F_0$, so
\[
\KS'(0,w)=\delta_p\bigl(dF_0(w)\bigr).
\]
Using
$\widetilde r_*\circ\delta_p
=\partial_N\circ\operatorname{pr}_N$
from Lemma~\ref{lem:34-analytic-blowup}, we therefore obtain
\begin{equation}\label{eq:34-moving-normal-direction}
\widetilde r_*\KS'(0,w)
=\partial_N\bigl(\nu(w)\bigr)\in \ker\gamma,
\end{equation}
where
$\nu=\operatorname{pr}_N\circ dF_0:T_0\Delta\to N|_p$. Since $dim(\ker\gamma)=1$ and $\partial_N\bigl(\nu(w)\bigr)\ne 0$ for $w\ne 0$, we have $L(\{0\}\oplus T_0\Delta)=\ker\gamma$. By \eqref{equa: surjection to H^1(C,N)}, $L$ is surjective, and the solution in $\bigcup_{(b,z)\in B\times\Delta}
\mathcal M^{\mathfrak s_{\alpha-e}}_{
\eta_\rho(\widetilde\omega_{b,z},\widetilde J_{b,z})}$ represented
by $\widetilde C$ is regular.
\end{proof}

\section{Local linking families and the proof of the main theorem}
\label{sec:Section 4}
Let $X$ be a $K3$ surface or $T^4$, and assume that $\kappa$ is non-resonant.
Throughout this section, take $1\le k\le n$ for K3 surfaces and
$2\le k\le n$ for tori, so that $r\ge2$. For a positive class $\kappa$, denote $\mathcal{X}_\kappa \rightarrow \mathcal{M}_\kappa$ the fine moduli space of K3 surfaces and $T^4$, and denote $X_u$ by the complex surface with the marking determined by $u\in\mathcal{M}_\kappa$. For $q=(u,\mathbf{x})\in\mathcal P_\kappa$, write $\tilde X_q=\operatorname{Bl}_{\mathbf{x}}X_u$. Take
$\lambda<\frac{\sqrt{\kappa^2}}{2n}$, set $\boldsymbol\lambda=(\lambda,\ldots,\lambda)$ and $S=n\lambda$, and put
\[
\alpha_{\kappa,\boldsymbol{\lambda}}=\kappa-\lambda\sum_{i=1}^n e_i,
 \qquad K=\sum_{i=1}^n e_i,
 \qquad I_k=\{1,\ldots,k\}.
\]
Recall the index set $\Delta_{X,k}$ in
Section~\ref{subsec:3.3}, for each $\delta\in\Delta_{X,k}$,
\begin{equation}\label{eq:4-area-window}
(k-1)\lambda<\kappa\cdot\delta<k\lambda,
\qquad \alpha_{\kappa,\boldsymbol{\lambda}}\cdot (\delta-\sum_{i\in I_k}e_i)<0.
\end{equation}
\subsection{Local incidence loci and transverse disks}
\label{subsec:4.1}
Let $\delta\in\Delta_{X,k}$, Since $\delta^2\le 0$ with $\kappa\cdot\delta\ne0$, the orthogonal projection $v\in \kappa^\perp\subset \Lambda_\mathbb{R}$ of $\delta$ satisfies $v^2<0$. By the properties of lattice $\Lambda_{\mathbb{R}}$, we can see that $H^\kappa_{\delta}:=H_\delta\bigcap \mathcal{M}_\kappa\subset\mathcal{M}_X$ is a nonempty hypersurface of $\mathcal{M}_\kappa$. Let $u_0$ be a generic point of it, i.e.
$\operatorname{NS}(X_{u_0})=\mathbb Z\delta$.
\begin{lem}\label{lem: curves representing delta}
There is a smooth connected holomorphic curve
$C_0\subset X_{u_0}$ representing $\delta$.
\end{lem}

\begin{proof}
Recall that $\operatorname{NS}(X_{u_0})=\mathbb Z\delta$
and that $\kappa$ is a K\"ahler class with
$\kappa\cdot\delta>0$.

Suppose first that $X_{u_0}$ is a K3 surface.
Since $\delta^2=-2$, either $\delta$ or $-\delta$ is
effective; see [\S3, Chapter~VIII, \cite{BarthPetersVandevenHulek2015}]. The latter does not happen since $\kappa\cdot\delta>0$. Since $\operatorname{NS}(X_{u_0})=\mathbb Z\delta$ and $\delta$ is primitive, any holomorphic curve $C$ representing $\delta$ must be reduced and irreducible. By  the adjunction formula, $p_a(C)=1+\delta^2/2=0$, $C_0:= C$ is a smooth rational curve.

Suppose now that $X_{u_0}$ is a complex torus $T^4$. Since $\delta$ is a primitive isotropic integral class of
type $(1,1)$, by the construction in the proof of
[Lemma~11, \cite{Smirnov2023}], there is a regular fibration
\begin{equation}\label{equa: fibration determined by delta}
\pi :X_{u_0}\longrightarrow E,
\end{equation}
over an elliptic curve $E$ such that each fiber of it is a connected smooth elliptic curve representing
$\pm\delta$. Since $\kappa\cdot\delta>0$, there is a smooth curve $C_0$ in $\delta$. 
\end{proof}

Recall the bound \eqref{equa : bound of area for negative curve classes}. By Lemma \ref{lem: connectivity of blowup-period domain}, hyperplanes $H^\kappa_a$ such that
\begin{equation}\label{eq:4-hyperplane-exclusion}
-c_X\leq a^2\leq0,
 \quad 0<\kappa\cdot a\leq(n+\sqrt{c_X n})\lambda,
\end{equation}
are locally finite. Here, $c_X=2$ for $K3$ surfaces and $c_X=0$ for $T^4$. Thus, we can take a holomorphic family $\{X_u|u\in U\}$ such that the base $U$ is a neighborhood of $u_0$, with $U\bigcap H^\kappa_a\ne \emptyset$ if and only if $H_a=H_\delta$, i.e. $a$ is equal to (a multiple of) $\delta$. Take a holomorphic disk $D_\delta\subset U$ with coordinate $z\in\mathbb D^2$, intersecting $H^\kappa_\delta$ once transversally at $u_0$, where $z=0$.

Recall from Section~\ref{subsec: n-point blowups of K3 and T^4} the
projection $r:\overline{\mathcal P}_\kappa
=\operatorname{Conf}_n(\mathcal X_\kappa/\mathcal{M}_\kappa)\longrightarrow \mathcal{M}_\kappa$, $\mathcal P_\kappa=\overline{\mathcal P}_\kappa$ for K3 surfaces and
$\mathcal P_\kappa=\{x_1=Z(u)\}\subset\overline{\mathcal P}_\kappa$ for tori.
We take the pullback family
\[
p:\mathcal P_{D_\delta}:=\mathcal P_\kappa\bigcap r^{-1}(D_\delta)\rightarrow D_\delta
\]
by the inclusion $D_\delta\hookrightarrow U\subset\mathcal{M}_\kappa$. When $X$ is a $K3$ surface, there is a unique smooth rational curve $C_0\subset X_{u_0}$ of self-intersection $-2$. When $X$ is a torus, there is a unique smooth elliptic curve $C_0$ on $X_{u_0}$ of self-intersection $0$ and contains $Z(u_0)$. For $q=(u,x_1,\ldots,x_n)\in\mathcal P_{D_\delta}$, we set $\tilde X_q$ to be the blowup of $X_u$ along $x_1,\cdots,x_n$. By the definition in \eqref{equa: P}, the fibers of $r|_{\mathcal P_\kappa}$ are of complex dimension $2n$
and $2n-2$ for $K3$ surface and $T^4$, respectively. Thus
\[
\dim_{\mathbb C}\mathcal P_{D_\delta}=
\begin{cases}
2n+1,&X\text{ a K3 surface},\\
2n-1,&X=T^4.
\end{cases}
\]
Define
\[
\mathcal H_{\delta}
:=\bigl\{(u_0,y_1,\ldots,y_n)\in\mathcal P_{D_\delta}|
y_i\in C_0\text{ for every }i\in I_k\bigr\}.
\]
$\mathcal H_{\delta}$ is the incidence locus for complex surfaces in $\tilde X_q$ such that there is a holomorphic curve in $\delta-\sum_{i\in I_k}e_i$. It is a holomorphic sub-manifold of $\mathcal{P}_{D_\delta}$. 

We say a point
$q_0=(u_0,x_1^0,\ldots,x_n^0)\in\mathcal H_\delta$ is generic if
\[
x_i^0\in C_0\quad (i\in I_k),
\qquad
x_j^0\notin C_0\quad (j\notin I_k),
\]
and when $X=T^4$, we further
require that the points $x_j^0$, $j\notin I_k$,
lie on pairwise distinct fibers of the elliptic
fibration associated with $\delta$, none of which
is $C_0$.

For K3 surfaces, $\mathcal{H}_{\delta}$ can be considered as an open subset of
$ C_0^k\times X_{u_0}^{n-k}$ consisting of pairwise distinct
points. For $X=T^4$, it is an open subset of $C_0^{k-1}\times X_{u_0}^{n-k}$
consisting of pairwise distinct points different from $Z(u_0)$. As a result, the dimension $\mathcal H_{\delta}$ is 
\[
\dim_{\mathbb C}\mathcal H_{\delta}=
\begin{cases}
k+2(n-k)=2n-k,&X\text{ a K3 surface},\\
(k-1)+2(n-k)=2n-k-1,&X=T^4.
\end{cases}
\]
Consequently,
\[
r:=\operatorname{codim}_{\mathbb C}
(\mathcal H_{\delta},\mathcal P_{D_\delta})
=\dim_{\mathbb C}\mathcal P_{D_\delta}
-\dim_{\mathbb C}\mathcal H_{\delta}
=\begin{cases}
k+1,&X\text{ a K3 surface},\\
k,&X=T^4.
\end{cases}
\]
Let $B_{\delta}$ be the image of a sufficiently small holomorphic ball embedding $i:\mathbb D^{2r}\hookrightarrow \mathcal{P}_{D_\delta}$ such that $\{q_0\}=B_\delta\bigcap \mathcal{H}_\delta$, $q_0$ is a generic point, and $B_\delta$ intersects $\mathcal{H}_\delta$ transversally. Suppose that the intersection point $q_0=(u_0,x_1,...,x_n)$ is a generic point of $\mathcal{H}_{\delta}$. Pull back the blowup family \eqref{eq:23-blowup-family} by $i$, and denote it by
\begin{equation}\label{equa: local universal family of B_delta,I}
p_{\delta}:\widetilde{\mathcal X}_{\delta}\to B_{\delta}
\end{equation} with blowdown maps
$\pi_q:\widetilde X_q\to X_{u(q)}$.

\subsection{K\"ahler structures on $B_\delta$ with described boundaries}
\label{subsec: 4.2}
In this section, we will use the parameter families $S_\delta$ to construct elements in $\pi_{2r-1}(\mathcal{S}_{\tilde\omega},\tilde{\omega})$. 
\begin{lem}\label{lem: kahler class on punctured B_delta,I}
The punctured ball $B_\delta\backslash\{q_0\}$ lies in $\mathcal P_{\kappa,\boldsymbol\lambda}$. In particular, $S_\delta=\partial B_\delta$ lies in $\mathcal P_{\kappa,\boldsymbol\lambda}$.
\end{lem}
\begin{proof}

By Lemma~\ref{lem: no bound from positive curves and small blowup},
if $\alpha_{\kappa,\boldsymbol\lambda}$ is not a K\"ahler class of $\tilde X_{q}$ for some
$q=(u,x_1,\ldots,x_n)\in B_\delta$, then there is a reduced
irreducible curve $C\subset X_u$ such that $\alpha_{\kappa,\boldsymbol\lambda}\cdot[\widetilde C]\le0$ and 
\[
-c_X\le [C]^2\le0,
\qquad
0<\kappa\cdot [C]\le (n+\sqrt{c_Xn})\lambda.
\]
By the local finiteness in Lemma \ref{lem: connectivity of blowup-period domain} and the choice of $D_\delta$ in Section \ref{subsec:4.1}, we may assume that 
\begin{enumerate}
\item $D_\delta\cap H_a=\emptyset$
for each $a\in\mathcal I_X(S)$
such that $H_a\ne H_\delta$,
\item $D_\delta\bigcap H_\delta=\{u_0\}$,
\item for each $i\notin I_k$ and each $(u_0,x_1,...,x_n)\in B_\delta$, $x_i\notin C_0$.
\end{enumerate}
If $u\ne u_0$, then the choice of $D_\delta$ implies that $u\notin H_a$ for each $a\in\mathcal{I}_X(S)$. Then by Lemma~\ref{lem: inclusion of blowup domains} and Lemma \ref{lem: connectivity of blowup-period domain},
$u\in\mathcal G_{\kappa,n\lambda}\subseteq\mathcal M_{\kappa,n\lambda} \subset\mathcal M_{\kappa,\boldsymbol\lambda}$, and $\alpha_{\kappa,\boldsymbol{\lambda}}$ is a K\"ahler class for any $q=(u,x_1,...,x_n)\in B_\delta$.

It remains to consider $u=u_0$. We distinguish the two cases.

Suppose first that $X_{u_0}$ is a K3 surface. Since
$\operatorname{NS}(X_{u_0})=\mathbb Z\delta$, every reduced
irreducible curve $C\subset X_{u_0}$ has class
$[C]=m\delta$ for some $m\in\mathbb Z$. Since $\kappa\cdot[C]>0$, we have $m>0$. By the adjunction formula, $m=1$ can be deduced from
$$p_a(C)=1-m^2\ge0.$$
Thus $C_0$ is the unique reduced irreducible
curve on $X_{u_0}$. Let $\tilde C_0$ be the proper transform of $C_0$ in $\tilde X_q$ and set
\[
m_{C_0}(q)
=\#\{i:x_i\in C_0\}.
\]
Since $C_0$ is smooth, the multiplicities at each point of it are equal to $1$, thus
\[
\alpha_{\kappa,\boldsymbol\lambda}\cdot[\widetilde C_0]
=\kappa\cdot\delta-\lambda m_{C_0}(q).
\]
By shrinking $B_\delta$ if necessary, $x_i\notin C_0$ for each $i\notin I_k$ and each $(u_0,x_1,...,x_n)\in B_\delta$. 
Since $B_\delta\bigcap \mathcal{H}_\delta=\{q_0\}$, $m_{C_0}(q)\le k-1$ for each $q\ne q_0$. By definition of $\Delta_{K3,k}$,
\[
\alpha_{\kappa,\boldsymbol\lambda}\cdot[\widetilde C_0]
\ge
\kappa\cdot\delta-(k-1)\lambda>0,
\]
and
$\alpha_{\kappa,\boldsymbol\lambda}$ is K\"ahler since it satisfies all conditions in Lemma \ref{lem: bound of blowup kahler class}.

Now suppose that $X_{u_0}$ is a complex two-torus. Since
$\operatorname{NS}(X_{u_0})=\mathbb Z\delta$, for every reduced
irreducible curve $C\subset X_{u_0}$, $[C]=m\delta$ for some $m>0$. Let $\pi$ be the holomorphic
elliptic fibration \eqref{equa: fibration determined by delta}
with fibers in class $\delta$, 
the image of $C$ under $\pi$ is a point since $C\cdot\delta=0$. Since $C$ is irreducible and contained in a fiber of $\pi$, it is a fiber of $\pi$, and consequently
$[C]=\delta$.

Recall the definition of $C_0$ and let $\pi(C_0)=p_0$ and let $U_{p_0}$ be a small neighborhood of $p_0$. By shrinking $B_\delta$ if necessary, $\pi(x_i)$ is contained in $U_{p_0}$ for each $i\in I_k$, whereas $\pi(x_i)\notin U_{p_0}$ for each $i\in \{1,...,n\}\backslash I_k$, and by our assumption, $\pi(x_i)\ne \pi(x_j)$ for $i\ne j\in \{1,...,n\}\backslash I_k$. For $q\ne q_0$, no fiber contains all $k$
points in $\{x_1,...,x_n\}$, since it happens only when
$q\in\mathcal H_\delta$. Thus
\[
\alpha_{\kappa,\boldsymbol\lambda}\cdot[\widetilde C]
=
\kappa\cdot\delta
-\lambda\#\{i:x_i\in C\}
\ge
\kappa\cdot\delta-(k-1)\lambda>0.
\]
Again, 
$\alpha_{\kappa,\boldsymbol\lambda}$ is K\"ahler since it satisfies all conditions in Lemma~\ref{lem: bound of blowup kahler class}. This completes the proof of the Lemma.
\end{proof}

\begin{lem}\label{lem: kahler-families with prescribed boundaries}
There is a smooth family $\{\tilde\omega_q| q\in B_\delta\}$ of K\"ahler forms such that
\begin{enumerate}
\item $[\tilde\omega_q]=\alpha_{\kappa,\boldsymbol{\lambda}}$
if $q$ is contained in a collar of $S_\delta$.
\item $\langle[\tilde\omega_q],\delta\rangle> 0$ for each $q\in B_\delta$.
\end{enumerate}
\end{lem}
\begin{proof}
By Lemma \ref{lem: kahler class on punctured B_delta,I}, for each $q\in B_\delta\backslash\{q_0\}$, there is a small neighborhood $U_q$ of $q$ and a smooth family $\{\tilde\omega_{q,q'}|q'\in U_q\}$ of symplectic forms such that $\tilde\omega_{q,q'}$ is a K\"ahler form of $\tilde X_{q'}$ in the class $\alpha_{\kappa,\boldsymbol{\lambda}}$. 

Since $\delta$ is represented by an effective divisor in $\tilde X_{q_0}$, there is a K\"ahler $\tilde\omega_{q_0}$ on $\tilde X_{q_0}$ such that $\langle\tilde\omega_{q_0},\delta\rangle>0$. And by the Kodaira-Spencer stability theorem \cite{KodairaSpencer1960}, there is a small neighborhood $U_{q_0}$ of $q_0$ such that $\tilde\omega_{q_0}$ extends to a smooth family $\{\tilde\omega_{q_0,q}|q\in U_{q_0}\}$ of K\"ahler forms with $\langle\tilde\omega_{q_0,q},\delta\rangle>0$ for each $q\in U_{q_0}$. Since $B_\delta$ is compact, there are finitely many points $\{q_1,...,q_m\}\subset B_\delta\backslash\{0\}$ such that $\mathcal{U}:=\{U_{q_0},U_{q_1},...,U_{q_m}\}$ is an open cover of $B_\delta$.

Choose a smooth partition of unity $\{\chi_0,\chi_1,...,\chi_m\}$ subordinate to $\mathcal{U}$.
By the construction
in [\S~3, \cite{Li2023}], define
\[
\tilde\omega_q=\sum_{i=0}^m \chi_i(q)\tilde\omega_{q_i,q},
\]
then $\tilde\omega_q$ is K\"ahler by convexity of the space of K\"ahler forms of a compact complex surface. Since $U_{q_0}$ is small, we have $[\tilde\omega_q]=\alpha_{\kappa,\boldsymbol{\lambda}}$ for each $q$ in a collar of $\partial B_\delta$. When $q\notin U_{q_0}$, the second condition is satisfied since $\langle\alpha_{\kappa,\boldsymbol\lambda},\delta\rangle>0$. When $q\in U_{q_0}$, since  $$\langle\alpha_{\kappa,\boldsymbol{\lambda}},\delta\rangle>0,\quad \langle\tilde\omega_{q_i,q},\delta\rangle>0,\quad \chi_i(q)\ge 0,\quad \sum_{i=0}^m\chi_i(q)\equiv 1,$$
we have 
\[
\langle\tilde\omega_q,\delta\rangle=\sum_{i=0}^m \chi_i(q)\langle\tilde\omega_{q_i,q},\delta\rangle>0
\]
and the second condition is satisfied.
\end{proof}
Choose $q_*=(u_*,\mathbf{x}_*)\in S_\delta$
with $u_*\ne u_0$. Since $B_\delta$ is contractible, there is a smooth
trivialization
$\Psi_q:\widetilde X_{q_*}\to\widetilde X_q$ such that
$\Psi_{q_*}=id$. And $\{\Psi_q^*\tilde\omega_q|q\in S_\delta\}$ defines an element of $\pi_{2r-1}(\mathcal{S}_{\tilde\omega_{q_*}},\tilde\omega_{q_*})$.

Let $(\widetilde X,\widetilde\omega)$ be as in
Theorem~\ref{thm: infinitely generated symp family}, we can identify to a K\"ahler form $\tilde\omega_q$ of $\tilde X_q$ where $q\in\mathcal P_{\kappa,\boldsymbol\lambda}$. Recall that $\mathcal{P}_{\kappa,\boldsymbol\lambda}\subset \overline{\mathcal{P}}_{\kappa,\boldsymbol\lambda}$ is a homological trivial family.
By Lemma~\ref{lem: connected kahler configuration locus}, there is a diffeomorphism
\[
h:(\widetilde X_q,\widetilde\omega_q)
\longrightarrow(\widetilde X_{q_*},\widetilde\omega_{q_*})
\]
preserving the cohomology markings. Define \[ \alpha_\delta:=[q\mapsto h^*\Psi_q^*\tilde\omega_q] \in\pi_{2r-1}(\mathcal S_{\tilde\omega},\tilde\omega). \]

\subsection{Evaluation of $q$-invariants and proof of Theorem \ref{thm: infinitely generated symp family}}
\label{subsec:4.3}
In this section, we evaluate $q_{\delta-\sum_{i\in I_k}e_i}(\alpha_{\delta'})\in\mathbb Z_2$ for $\delta,\delta'\in\Delta_{X,k}$, and prove the main theorem.

Using $\Psi_q\circ h:\widetilde X\to\widetilde X_q$, we pull the complex structures and K\"ahler forms back to $\widetilde X$, keeping the notation $\tilde J_q$ and $\tilde\omega_q$ for $q\in B_\delta$.
\begin{lem}\label{lem: counting of FSW in B_delta}
Set $\epsilon=\delta-\sum_{i\in I_k}e_i$, then
\[
Q_{\epsilon}(\alpha_\delta)
=\FSW_{\mathfrak s_\epsilon}\bigl(
\widetilde X\times B_\delta,
\mathfrak T_\rho(\widetilde\omega_q,\widetilde J_q)
\bigr)
=1\in\mathbb Z_2
\]
for all sufficiently large $\rho$.
\end{lem}

\begin{proof}
For simplicity, we write $\alpha=\alpha_{\kappa,\boldsymbol\lambda}$ and
$g_q=\widetilde\omega_q(\cdot,\widetilde J_q\cdot)$.
Both $\kappa$ 
and each exceptional class $e_i$ are contained in $H^{1,1}(\tilde X_q;\mathbb R)$, so $\alpha$ and $[\widetilde\omega_q]$ both belong to
$H^{1,1}(\widetilde X_q;\mathbb R)$ for every $q\in B_\delta$. Since $D_\delta\subset \mathcal{M}_\kappa$, $\kappa$ is contained in the positive cone of $H^{1,1}(\tilde X_q;\mathbb R)$, and 
\[
\alpha^2>0,\qquad [\widetilde\omega_q]^2>0,\qquad
\alpha\cdot\kappa=\kappa^2>0,\qquad
[\widetilde\omega_q]\cdot\kappa>0.
\]
Then by the Hodge index theorem, $\alpha$ and
$[\widetilde\omega_q]$ lie in the same component of the positive
cone of $H^{1,1}(\widetilde X_q;\mathbb R)$. Consequently,
\[
\alpha\cdot[\widetilde\omega_q]>0
\quad \text{ for each }q\in B_\delta.
\]
By [Lemma $4.15$, \cite{DuLi2026I}], $Pr(\mathfrak T_\rho(\widetilde\omega_q,\widetilde J_q))$ will be contained in $\mathbf K_{-\alpha}(\widetilde X)$ for all sufficiently large $\rho$.
By Lemma~\ref{lem: admissible extension},  $\{\mathfrak T_\rho(\widetilde\omega_q,\widetilde J_q)|q\in B_\delta\}$ is an admissible extension of its
restriction to $S_\delta$, and the first equality holds. 

Clearly, there is a curve $\tilde C_0$ in the class $\delta-\sum_{i\in I_k}e_i$ contained in the surface $\tilde X_{q_0}$. Conversely, we suppose that $D\subset\widetilde X_q$ is an effective divisor
in class $\epsilon$, where $q=(u,x_1,\ldots,x_n)$. Let $\pi:\widetilde X_q\to X_u$ be the natural projection, then $\pi_*(D)$ is an
effective divisor of $X_u$ in class $\delta$. Since
$D_\delta\cap H_\delta=\{u_0\}$, we must have $u=u_0$. When $X$ is a K3 surface, its pushforward is $C_0$; if $X$ is a torus, the pushforward is a smooth fiber of the elliptic fibration \eqref{equa: fibration determined by delta}. Denote it by $C$, then $x_i\in C$ for each $i\in I_k$.
In the torus case, $x_1=Z(u_0)\in C$ implies that $C=C_0$. By the construction of $B_\delta$, we can see that $q\in B_\delta\bigcap \mathcal{H}_\delta$, hence $q=q_0$ and $D=\widetilde C_0$. By \cite{Taubes1995}\cite{FriedmanMorgan1995}, $\bigcup_{q\in B_\delta}\mathcal{M}^{\mathfrak s_\epsilon}_{\mathfrak T_\rho(\tilde\omega_q,\tilde J_q)}$ consists of a unique point, which is corresponding to $\tilde C_0$ in $\tilde X_{q_0}$.

It remains to prove regularity of the point in $\bigcup_{q\in B_\delta}\mathcal{M}^{\mathfrak s_\epsilon}_{\mathfrak T_\rho(\tilde\omega_q,\tilde J_q)}$. Set $a=0$ for K3 surfaces
and $a=1$ for tori. Over $B^{(a)}=D_\delta$, we take the following family:
\[
\mathcal X^{(a)}=
\begin{cases}
\mathcal X_\kappa|_{D_\delta},&X\text{ is a K3 surface},\\
\operatorname{Bl}_{Z(D_\delta)}
\mathcal X_\kappa|_{D_\delta},&X=T^4.
\end{cases},\quad C_{a}=\begin{cases}
C_0,&X\text{ is a K3 surface},\\
\text{proper transform of } C_0,&X=T^4.
\end{cases}
\]
By
[\S~8, \cite{Smirnov2022}] and
[\S~9, \cite{Smirnov2023}], $res\circ \KS: T_0D_\delta\rightarrow H^1(C_a,N_{C_a/\mathcal X^{(a)}_{u_0}})$ is surjective, respectively.

For $a<j\le k$, we take jointly holomorphic maps
\[
F_j:D_\delta\times\Delta_j
\longrightarrow\mathcal X_\kappa|_{D_\delta},
\qquad
p_\kappa\circ F_j=\operatorname{pr}_{D_\delta},
\qquad
F_j(u_0,0)=x_j(q_0),
\]
such that $\{F_j(u_0,z)|z\in \Delta_j\}$ is transverse to $C_0$ for each $j$ and the images of $F_j$ are disjoint to each other, and disjoint from the zero section $Z$ in the torus case. Here, $x_j(q_0)$ is the coordinate of $q_0$. By shrinking $\Delta_j$ if necessary, we assume that $F_j(u_0,\cdot)^{-1}(C_0)=\{0\}$.
For each $i>k$, choose a holomorphic section
$s_i:D_\delta\to\mathcal X_\kappa|_{D_\delta}$
with $s_i(u_0)=x_i(q_0)$. By shrinking the image of $s_i$,
we may again assume that the images are pairwise disjoint, disjoint from $F_j$ and, for tori, the zero section $Z$. Set
\[
P_\delta
=D_\delta\times\prod_{i=a+1}^k\Delta_i,
\qquad
B^{(j)}
=D_\delta\times\prod_{i=a+1}^j\Delta_i.
\]
The holomorphic maps $F_i$, the sections $s_i$, together with the section $Z$ in
the torus case, determine
a holomorphic map
\[
\Psi:P_\delta\longrightarrow\mathcal P_{D_\delta},
\qquad
(u,\mathbf z)\longmapsto
(u,x_1(u,\mathbf z),\ldots,x_n(u,\mathbf z)),
\]
\[
x_i(u,\mathbf z)=
\begin{cases}
F_i(u,z_i),&a<i\le k,\\
Z(u),&X=T^4,\ i=1,\\
s_i(u),&i>k.
\end{cases}
\]
In particular, $\Psi(u_0,0)=q_0$. 
For $a+1\le j\le k$, suppose inductively that
the holomorphic family
\[
p^{(j-1)}:\mathcal X^{(j-1)}
\longrightarrow B^{(j-1)}
\]
has been constructed, starting from
$\mathcal X^{(a)}$ as above. Its fiber over $(u,z_{a+1},\ldots,z_{j-1})$
is obtained from blowing up each fiber of $D_\delta$ by successively blowing up the points 
$$F_{a+1}(u,z_{a+1}),\ldots,F_{j-1}(u,z_{j-1})$$
simultaneously (also blowup along $Z(u)$ if $X$ is torus). Let
\[
\Pi_{j-1}:\mathcal X^{(j-1)}
\longrightarrow
B^{(j-1)}\times_{D_\delta}(\mathcal X_\kappa|_{D_\delta})
\]
be the corresponding blowdown map. Set
$\mathcal Y^{(j)}
=B^{(j)}\times_{B^{(j-1)}}\mathcal X^{(j-1)}$, since the image of $F_j$ is disjoint from the
images of $F_{a+1},...,F_{j-1}$ (and $Z(u)$ if $X=T^4$), the lift of $F_j$ under
$\Pi_{j-1}$ defines a holomorphic section
\[
\sigma_j:B^{(j)}\longrightarrow\mathcal Y^{(j)},
\]
and $p^{(j)}$ is defined as follows
\[
p^{(j)}:\mathcal X^{(j)}
=\operatorname{Bl}_{\sigma_j(B^{(j)})}
\mathcal Y^{(j)}\longrightarrow B^{(j)}.
\]


Let $X_0^{(j)}$ be the central fiber of $\mathcal X^{(j)}$,
and let $C_j\subset X_0^{(j)}$ be the proper transform of
$C_a$. Since the image of $F_j$ are disjoint from each other (and $Z(u)$ when $X=T^4$), and $\{F_j(u_0,z)|z\in\Delta_j\}$ is transverse to $C_0$ for each $j$, then the lift of $\{F_j(u_0,z)\}$ remains transverse
to $C_{j-1}$. Together with the surjectivity of $res\circ \KS: T_0D_\delta\rightarrow H^1(C_a,N_{C_a/\mathcal X^{(a)}_{u_0}})$, we can apply Theorem~\ref{thm:regularity of FSW solutions on blow up family}
at each step and get the surjectivity of
\[
\operatorname{res}\circ\KS_j:
T_0B^{(j)}
\longrightarrow H^1(C_j,N_{C_j/X_0^{(j)}})
\]
for every $a\le j\le k$.

For each $i>k$, $s_i$ are disjoint from $C_0$, blowing up $\mathcal X^{(k)}$ along the lifts of $s_i$, the blowdown map of each fiber is a biholomorphism near $C_k$ and hence preserve
its normal bundle and the restriction of the Kodaira--Spencer map. Since the image of $s_i$ are pairwise disjoint and disjoint from $F_j$ and, for tori, $Z(u)$, the resulting family can be considered as a pullback of the blowup family
\eqref{eq:23-blowup-family} by $\Psi$. Therefore, the surjectivity obtained above implies that
\[
\operatorname{res}\circ\KS:
T_{q_0}\mathcal P_{D_\delta}
\longrightarrow
H^1(\widetilde C_0,N_{\widetilde C_0/\widetilde X_{q_0}})
\]
is surjective.

Finally, the restriction of $\widetilde p:\widetilde{\mathcal X}\rightarrow\overline{\mathcal P}_\kappa$
to a neighborhood of $q_0$ in $\mathcal H_\delta$ is a deformation of the pair
$(\widetilde X_{q_0},\widetilde C_0)$, hence $T_{q_0}\mathcal H_\delta
\subseteq\Ker(\operatorname{res}\circ\KS)$. Since
\[
h^1(\widetilde C_0,N_{\widetilde C_0/\widetilde X_{q_0}})
=\operatorname{codim}_{\mathbb C}
(\mathcal H_\delta,\mathcal P_{D_\delta})=r,
\]
and $B_\delta$ is transverse to $\mathcal H_\delta$,
we conclude that
\[
\operatorname{res}\circ\KS|_{B_\delta}:
T_{q_0}B_\delta
\xrightarrow{\;\cong\;}
H^1(\widetilde C_0,N_{\widetilde C_0/\widetilde X_{q_0}}).
\]
By \cite{FriedmanMorgan1995}\cite{Kronheimer1998}, the solution in $\bigcup_{q\in B_\delta}\mathcal{M}^{\mathfrak s_\epsilon}_{\mathfrak T_\rho(\tilde\omega_q,\tilde J_q)}$ represented by $\widetilde C_0$ is regular, and
\[
Q_\epsilon(\alpha_\delta)
=\FSW_{\mathfrak s_\epsilon}\bigl(
\widetilde X\times B_\delta,
\mathfrak T_\rho(\widetilde\omega_q,\widetilde J_q)
\bigr)=1\in\mathbb Z_2.
\]
\end{proof}
\begin{lem}\label{lem: vanishing of Q_K-delta+sum e_i}
$Q_{K-(\delta-\sum_{i\in I_k}e_i)}(\alpha_\delta)=Q_{\sum_{j\in \{1,...,n\}\backslash I_k}e_j+2\sum_{i\in I_k}e_i-\delta}(\alpha_\delta)=0$.
\end{lem}
\begin{proof}
The admissibility argument in
Lemma~\ref{lem: counting of FSW in B_delta} applies equally to
$\mathfrak s_{K-\epsilon}$ for sufficiently large $\rho$, and solutions in $\bigcup_{q\in B_\delta}\mathcal{M}^{\mathfrak s_{K-\epsilon}}_{\mathfrak T_\rho(\tilde\omega_q,\tilde J_q)}$ correspond to effective
divisors in class $K-\epsilon$.
By [Theorem $1.3$, \cite{Taubes1996}], if $Q_{K-(\delta-\sum_{i\in I_k}e_i)}(\alpha_\delta) \ne 0$, then there is $q\in B_\delta$ and a curve $\tilde C\subset \tilde X_q$ such that $\PD[\tilde C]= K-(\delta-\sum_{i\in I_k}e_i)$. As a result, $-\delta$ is an effective class in the minimal model of $\tilde X_q$. However, the minimal model of $\tilde X_q$ is $X_u$, where $u\in \mathcal{M}_\kappa$. This contradicts the fact that $\langle\kappa,\delta\rangle>0$.

\end{proof}
By the local finiteness of hyperplanes in $\mathcal{I}_X(S)$ and the construction of $D_{\delta'}$, $\delta'\notin H^{1,1}(\tilde{X}_q,\mathbb R)$ for each $\delta'\in\Delta_{X,k}\backslash\{\delta\}$ and $q\in B_{\delta}$. For either class $\epsilon'=\delta'-\sum_{i\in I_k}e_i$ or $K-\epsilon'$,
the same admissibility argument applies to the K\"ahler family over
$B_{\delta}$. Since $\delta'$ is not of $(1,1)$-type in each fiber of $B_\delta$, and corresponding moduli spaces are empty by [Theorem $1.3$, \cite{Taubes1996}]. Thus, $Q_{\delta'-\sum_{i\in I_k}e_i}(\alpha_{\delta})=Q_{K-(\delta'-\sum_{i\in I_k}e_i)}(\alpha_{\delta})=0$ if $\delta'\ne \delta\in\Delta_{X,k}$. Thus we have
\begin{equation}\label{equa: coordinate of delta-sum e_i}
q_{\delta'-\sum_{i\in I_k}e_i}(\alpha_{\delta})
=Q_{\delta'-\sum_{i\in I_k}e_i}(\alpha_{\delta})=\begin{cases}1,&\delta'=\delta,\\0,&\delta'\ne\delta.\end{cases}
\end{equation}
Set
\[
G=\Ker\!\left(
\pi_{2r-2}(\Symp_{st}(\widetilde X,\widetilde\omega))
\longrightarrow
\pi_{2r-2}(\Diff_0(\widetilde X))
\right).
\]
For each $\delta\in\Delta_{X,k}$, write
$\epsilon_\delta=\delta-\sum_{i\in I_k}e_i$ and
$\beta_\delta=\partial_*\alpha_\delta\in G$.
By Lemma~\ref{lem: reduction of q to monodromies} and
\eqref{equa: coordinate of delta-sum e_i},
\[
\bar q_{\epsilon_\delta}(\beta_{\delta'})
=\begin{cases}
1,&\delta=\delta',\\
0,&\delta\ne\delta'.
\end{cases}
\]
Consequently, for every finite subset $F\subset\Delta_{X,k}$,
the homomorphism
\[
(\bar q_{\epsilon_\delta})_{\delta\in F}:
G\longrightarrow\bigoplus_{\delta\in F}\mathbb Z_2
\]
is surjective. Since $\Delta_{X,k}$ is infinite, $G$ is not finitely
generated, proving Theorem~\ref{thm: infinitely generated symp family}.
\bibliographystyle{plain} 
\bibliography{reference}

\end{document}